\documentclass[12pt]{amsart}
\usepackage{amssymb, eucal, amsfonts, amsmath, xy,latexsym, bbm}
\usepackage{mathrsfs}

\usepackage[margin=2cm]{geometry}
\usepackage[pdftex,colorlinks,hypertexnames=false,linkcolor=black,urlcolor=black,citecolor=black]{hyperref}

\numberwithin{equation}{section}

\newtheorem{Theorem}{Theorem}[section]
\newtheorem*{Theorem*}{Theorem}
\newtheorem*{Corollary*}{Corollary}

\newtheorem{Lemma}[Theorem]{Lemma}
\newtheorem{Proposition}[Theorem]{Proposition}

\newtheorem*{Conjecture}{Conjecture}

\theoremstyle{definition}

\theoremstyle{remark}
\newtheorem{Remark}[Theorem]{Remark}
\newtheorem*{Remark*}{Remark}

\newcommand{\C}{\mathbb{C}}
\newcommand{\N}{\mathbb{N}}
\newcommand{\Z}{\mathbb{Z}}
\newcommand{\Q}{\mathbb{Q}}
\renewcommand{\k}{\mathbbm{k}}
\renewcommand{\Z}{\mathbb{Z}}

\newcommand{\Sb}{\mathbb{S}}

\newcommand{\g}{\mathfrak{g}}
\newcommand{\p}{\mathfrak{p}}
\renewcommand{\v}{\mathfrak{v}}
\newcommand{\gl}{\mathfrak{gl}}

\renewcommand{\c}{\mathfrak{c}}
\renewcommand{\t}{\mathfrak{t}}
\renewcommand{\l}{\mathfrak{l}}
\renewcommand{\sl}{\mathfrak{sl}}
\newcommand{\so}{\mathfrak{so}}
\renewcommand{\sp}{\mathfrak{sp}}
\newcommand{\z}{\mathfrak{z}}
\newcommand{\m}{\mathfrak{m}}
\newcommand{\n}{\mathfrak{n}}

\newcommand{\q}{\mathfrak{q}}

\newcommand{\Nc}{\mathcal{N}}

\renewcommand{\O}{\mathcal{O}}

\newcommand{\I}{\mathcal{I}}

\renewcommand{\P}{{\mathcal{P}}}

\newcommand{\F}{\mathcal{F}}

\newcommand{\ad}{\operatorname{ad}}
\newcommand{\Ad}{\operatorname{Ad}}

\newcommand{\Lie}{\operatorname{Lie}}

\newcommand{\im}{\operatorname{Im}}

\newcommand{\GL}{\operatorname{GL}}

\newcommand{\Sp}{\operatorname{Sp}}

\newcommand{\Aut}{\operatorname{Aut}}

\newcommand{\End}{\operatorname{End}}

\newcommand{\Mat}{\operatorname{Mat}}

\newcommand{\Ind}{\operatorname{Ind}}

\newcommand{\Char}{\operatorname{char}}

\newcommand{\op}{{\operatorname{op}}}

\newcommand{\col}{\operatorname{col}}
\newcommand{\row}{\operatorname{row}}
\newcommand{\refl}{\operatorname{r}}

\newcommand{\ve}{\varepsilon}

\newcommand{\ug}{\underline{\g}}
\newcommand{\ue}{\underline{e}}
\newcommand{\uchi}{\underline{\chi}}
\newcommand{\uG}{\underline{G}}
\newcommand{\uO}{\underline{\O}}

\newcommand{\hZ}{\widehat{Z}}

\renewcommand{\i}{{\bf i}}
\renewcommand{\j}{{\bf j}}
\newcommand{\PeN}{\P_\epsilon(N)}

\begin{document}
\title[Modular representations and Humphreys' conjecture]{Modular representations of Lie algebras of reductive groups and Humphreys'
conjecture}

\author{Alexander Premet and Lewis Topley}

\thanks{\nonumber{\it Mathematics Subject Classification} (2000 {\it revision}).
Primary 17B35, 17B50. Secondary 17B20.}
\address{Department of Mathematics, The University of Manchester, Oxford Road, M13 9PL, UK} 
\email{alexander.premet@manchester.ac.uk}
\address{School of Mathematics, University of Birmingham,
Edgbaston, Birmingham B15 2TT, UK}\email{L.Topley.1@bham.ac.uk} \maketitle
\begin{center}
	{\it To the memory of James E. Humphreys}
\end{center}

\begin{abstract}
\noindent Let $G$ be connected reductive algebraic group defined over an algebraically 
closed field of characteristic $p>0$ and suppose that  $p$ is a good prime for the root system of $G$, the derived subgroup of $G$ is simply connected 
and the Lie algebra $\g=\Lie(G)$ admits a non-degenerate  $(\Ad\,G)$-invariant symmetric bilinear form. Given a linear function $\chi$  on $\g$ we denote by $U_\chi(\g)$ 
the reduced enveloping algebra of $\g$ associated with $\chi$.
By the Kac--Weisfeiler conjecture (now a theorem), any irreducible $U_\chi(\g)$-module has dimension divisible by
$p^{d(\chi)}$ where $2d(\chi)$ is the dimension of the coadjoint $G$-orbit containing $\chi$. In this paper we give a positive answer to the natural question raised in the 1990s by Kac, Humphreys and the first-named author and show that any algebra
$U_\chi(\g)$ admits a module of dimension $p^{d(\chi)}$. 
\end{abstract}
\maketitle

\section{Introduction}
Let $\k$ be an algebraically closed field of characteristic $p>0$ and let $G$ be a connected reductive $\k$-group. The Lie algebra $\g=\Lie(G)$ carries a canonical $[p]$-mapping $\g\ni x\mapsto x^{[p]}\in\g$ equivariant under the adjoint action of $G$ on $\g$. In this paper we always assume that $G$ is {\it standard}, that is, the derived subgroup of $G$ is simply connected, $p$ is a good prime
for the root system of $G$, and $\g$ admits a non-degenerate $G$-invariant symmetric bilinear form. 

Let $U(\g)$ denote the universal enveloping algebra of $\g$. It is well-known that all irreducible $\g$-modules are finite-dimensional and to any irreducible $\g$-module $V$ one can attach a linear
function $\chi\in \g^*$, called the $p$-{\it character} of $V$, such that for any $x\in \g$ the element $x^p-x^{[p]}\in U(\g)$ acts on $V$ as
multiplication by $\chi(x)^p$. Conversely, any linear function on $\g$ serves as a $p$-character of an irreducible $\g$-module. This follows from the fact that for any $\chi\in\g^*$ the $\g$-modules with $p$-character $\chi$ are precisely the modules over the {\it reduced enveloping algebra} $U_\chi(\g)\,:=\,U(\g)/I_\chi$ where $I_\chi$ is the two-sided ideal of $U(\g)$ generated by all $x^p-x^{[p]}-\chi(x)^p$ with $x\in \g$. 

By the Kac--Weisfeiler conjecture \cite{KW71} first proved in
\cite{Pr95} the dimension of any $U_\chi(\g)$-module is divisible by $p^{(\dim\O(\chi))/2}$ where $\O(\chi)$ denotes the coadjoint $G$-orbit of $\chi$ (different proofs were later found in \cite{PS99} and \cite{BMR}).
Since a priori it is unclear that this $p$-divisibility constraint is sharp,  a question was raised in \cite[p.~114]{Pr95} whether for any linear function $\chi\in \g^*$ there exists an irreducible $\g$-module with $p$-character $\chi$ whose dimension equals $p^{(\dim\O(\chi))/2}$. This question was explicitly mentioned by Kac \cite{Kac} and Humphreys  \cite[p.~110]{Ha} and became known as {\it the problem of small representations} or simply as {\it Humphreys' conjecture}. In this context, a representation $\rho\colon\, U_\chi(\g)\to \End(V)$ is called {\it small} if
$\dim V=p^{(\dim\O(\chi))/2}$.

There are some long-standing problems in the characteristic-zero theory of primitive ideals which have the same flavour as the problem of small representations. Here one should mention the problem of constructing and classifying all completely prime primitive ideals of $U(\g_\C)$ with a prescribed associated variety, where $\g_\C$ is a finite-dimensional simple Lie algebra over $\C$. Experts know that such problems are much easier to tackle in type $\rm A$. So it is not surprising that the problem of small modular representations has a positive solution for $\g=\mathfrak{gl}_n(\k)$.
This has been known for some time, but very recently Goodwin and the second-named author managed to classify all such representations. Specifically, they proved that all small $\mathfrak{gl}_n(\k)$-modules with $p$-character $\chi$ can be obtained by induction from one-dimensional representations of a suitable parabolic subalgebra of $\mathfrak{gl}_n(\k)$; see \cite{GT19a}. This result can be regarded as a modular analogue of M{\oe}glin's classification of the completely prime primitive ideals of $U(\gl_n(\C))$.

We stress that in general small $U_\chi(\g)$-modules $p$-character $\chi$ are difficult to construct and even harder to classify. 
Indeed, it can be rigorously proved that if $\chi\ne 0$ is {\it rigid}  (i.e. non-induced in the sense of Lusztig--Spaltenstein) then small $U_\chi(\g)$-modules with $p$-character $\chi$ cannot be obtained by parabolic induction from proper Levi subalgebras of $\g$.
Nevertheless, the existence of small modular representations was established  in a handful of rigid cases by Kac--Radul, Chari--Pressley, Peters, Peters--Shi and Jantzen; see \cite{Kac, Ch-P, Pet, P-Shi, Ja97}.

The main result of this paper is the following.
\begin{Theorem}\label{main}
Let $G$ be a standard reductive $\k$-group and $\g=\Lie(G)$. Then for any $\chi\in \g^*$ there exists an irreducible $U_\chi(\g)$-module of dimension $p^{d(\chi)}$ where $d(\chi)$ is half the dimension of the coadjoint orbit
$G$-orbit of $\chi$.  
\end{Theorem}

The forthcoming paper of Goodwin--Topley builds in a crucial way on the results and methods of the present work and gives a partial solution to the classification of small modules in symplectic and orthogonal Lie algebras $\g=\Lie(G)$ under the assumption that $p\ne 2$. It will be shown that if $\chi$ lies in a unique sheet then any small module $U_\chi(\g)$-module is parabolically induced from a small module over a Levi subalgebra $\l$ of $\g$ such that the restriction of $\chi $ to $\l$ is rigid. For a general nilpotent $\chi\in\g^*$ there is a natural action of the component group of $G^\chi$ on the set of all isomorphism classes of irreducible $U_\chi(\g)$-modules, and it will be demonstrated that every small $U_\chi(\g)$-module fixed by this action is parabolically induced from a small $\l$-module whose $p$-character $\chi\vert_{\l}$ is rigid in $\l^*$. The latter result may be regarded as a modular analogue of the description of multiplicity-free primitive ideals of $U(\so_N(\C))$ and $U(\sp_N(\C))$ obtained in \cite{PT14}.

In order to prove Theorem~\ref{main} we first reduce the general problem to the case where the group $G$ is simple and the linear function $\chi$ is $G$-unstable in the sense of Geometric Invariant Theory. This is done by using a Morita theorem for $U_\chi(\g)$ proved by Friedlander-Parshall and some results on the structure of $\g$ obtained in \cite[2.1]{PS18} (see also \cite[Proposition~2.5]{Sasha1}). As $\g$ admits a non-degenterate $G$-invariant bilinear form we can identify a $G$-unstable $\chi\in\g^*$ with a nilpotent element $e$ of $\g$. If this element can be obtained by Lusztig--Spaltenstein induction from a proper Levi subalgebra $\l$ of $\g$ we construct a small $U_\chi(\g)$-module by parabolic induction from a small $U_{\bar{\chi}}(\l)$-module where $\bar{\chi}=\chi_{\vert\l}$; see Subsection~\ref{ss:smallrigid}. We thus reduce the problem of small representations to the case where $G$ is simple and $e$ is a nonzero rigid nilpotent element of  $\g$ (the latter means that the sheet of $\g$ containing $e$ coincides with the adjoint $G$-orbit of $e$). In particular, we may assume that $G$ is not of type $\rm A$. 

Since $p$ is a good prime for the root system of $G$, the rigid nilpotent orbits of $\g$ have the same description as in the characteristic zero case. For $G$ exceptional this is proved in \cite[Theorem~3.8]{PS18}. For classical groups over algebraically closed fields of characteristic $\ne 2$ a uniform combinatorial description of rigid nilpotent orbits was obtained by Kempken \cite{K} and Spaltenstein \cite{Sp1}. 

Next we apply some results obtained in \cite{Sasha1, Pr10} and  \cite{GT18} to identify the algebra $U_\chi(\g)$ with the matrix algebra ${\rm Mat}_{p^{d(\chi)}}(U_\chi(\g,e))$ where $$U_\chi(\g,e)\,=\,U(\g,e)/J_\chi U(\g,e)$$ is a $p$-central reduction of the modular finite $W$-algebra $U(\g,e)$; see Subsection~\ref{ss:pcentreandSkryabin} for more detail. We then use a contracting $\k^\times$-action on a good slice to the adjoint $G$-orbit of $e$ to show that if the modular finite $W$-algebra $U(\g,e)$ admits a one-dimensional representation then so does its finite-dimensional quotient
$U_\chi(\g,e)$. We thus reduce Humphreys' conjecture to proving that all modular finite $W$-algebras $U(\g,e)$ associated with nonzero rigid nilpotent elements
of $\g$ contain ideals of codimension $1$ (and hence afford  one-dimensional representations).

It was conjectured in \cite{Pr07a} that in the characteristic-zero case
all finite $W$-algebras afford one-dimensional representations. This conjecture was reduced to the rigid case in \cite{Pr10} and then confirmed for classical groups in \cite{PT14}. The case of exceptional groups was dealt with  in \cite{GRU} and \cite{Pr14}.
The arguments in \cite{PT14, Pr14} relied on important results proved by Losev in \cite{Lo3, Lo5} whilst \cite{GRU} was entirely based on computer-aided computations.
In order to make use of these results we now have to link the rigid modular finite $W$-algebras with their characteristic-zero counterparts. 

We denote by $R$ the localisation of $\Z$ at the bad primes of the root system $\Phi$ of $G$. Thus we set $R\,:=\,\Z[\frac{1}{2}]$ for $G$ of type $\rm B$, $\rm C$ or $\rm D$, $R\,:=\,\Z[\frac{1}{6}]$ for $G$ of type ${\rm G}_2$, ${\rm F}_4$, ${\rm E}_6$, ${\rm E}_7$, and $R\,:=\,\Z[\frac{1}{30}]$ for $G$ of type ${\rm E}_8$. Let $\g_\Z$ denote the Chevalley $\Z$-form of the complex Lie algebra
$\g_{\C}$ with root system $\Phi=\Phi(G)$ and set $\g_R\,:=\,\g_\Z\otimes_\Z R$. As $p$ is a good prime for $\Phi$ and the group $G$ is not of type $\rm A$ we  may assume that $\g\,=\,\g_{R}\otimes_R \k$
and $e=\tilde{e}\otimes_R 1$ for some nilpotent element $\tilde{e}\in  \g_R$. In view of \cite{Pr03} and \cite{PS18} we may assume that the GIT-unstable vectors $e$ and $\tilde{e}$ share the same optimal cocharacter $\lambda_e\colon \mathbb{G}_{\rm m}\to\, G_\Z$ where $G_{\Z}$ is a Chevalley group scheme
with root system $\Phi$. Furthermore, we may suppose that both $\tilde{e}$ and $e$ have weight $2$ with respect to $\lambda_e$. Our choice of $\tilde{e}$ is specified in Subsection~\ref{ss:char0nilp} where we invoke basic notions of the Bala--Carter theory.
By \cite[Theorem~3.8]{PS18}, if $e$ is rigid then so is $\tilde{e}$.
\begin{Theorem}\label{main1}
The following are true under the above assumptions on the rigid nilpotent elements $e\in\g$ and $\tilde{e}\in\g_R$.
\begin{enumerate} 
\item 
The finite $W$-algebra $U(\g_\C,\tilde{e})$ contains a unital $R$-subalgebra
$U(\g_R,\tilde{e})$ which is freely generated as an $R$-module by a $PBW$ basis of $U(\g_\C,\tilde{e})$. Moreover, $$U(\g_R,\tilde{e})\otimes_R \C\,\cong\,U(\g_\C, \tilde{e})\ \  
\mbox{and}\ \  
U(\g_R,\tilde{e})\otimes_R\k \,\cong\,U(\g, e)$$ as algebras over the respective fields.

\smallskip

\item   The $R$-algebra
$U(\g_R,\tilde{e})$ contains a two-sided ideal $I_R$ which is a free  $R$-module and has the property that  $U(\g_R,\tilde{e})\,=\,R\, 1\oplus I_R$.

\smallskip

\item The two-sided ideal $I_R\otimes_R\k$ of $U(\g,e)\,\cong\, U(\g_R,\tilde{e})\otimes_R\k$ has codimension $1$ in $U(\g,e)$.
\end{enumerate}
\end{Theorem}
To be more precise, for any nilpotent element $\tilde{e}$ as above we construct
an $R$-form $Q_R$ of the generalised Gelfand--Graev module $Q_\C$ associated with $\tilde{e}$ and then show, for $\tilde{e}$ rigid, that the endomorphism ring $U(\g_R,\tilde{e})\,:=\,{\rm End}_{\g_R}(Q_R)^{\rm op}$ has a nice PBW basis.
Our proof of Theorem~\ref{main1} relies on some structural properties of the centraliser $\g^e$ established in \cite{Ya10}, \cite{PT14} and \cite{PS18}. In particular, we use the fact that either $\g^e=[\g^e,\g^e]$ or $\g^e=\k e\oplus [\g^e,\g^e]$ and the latter occurs for the six rigid orbits listed in Table~1.
As a consequence, we show that in all rigid cases the small $\g$-modules with $p$-character
$\chi$ are separated by the action of a Casimir element of $U(\g)$. 
When $\g^e=[\g^e,\g^e]$ the Lie algebra $\g$ affords a unique small module with $p$-character $\chi$.

\begin{Conjecture}
We conjecture that Theorem~\ref{main1} holds for all nilpotent elements  $e\in\g$. 
\end{Conjecture}
It seems likely that in order to prove this conjecture one would have to check that the arguments used by Gan--Ginzburg in \cite{GG02} go through over the ring $R$. It follows from the results of \cite{BK} and \cite[3.4]{GT19b}
that the conjecture does hold for $\gl_n$.

If $\tilde{e}$ is rigid and special in the sense of Lusztig then  \cite[Theorem~B]{Pr14} implies that the central character of the unique small $U_\chi(\g)$-module corresponds to the orbit of the Arthur--Barbasch--Vogan weight associated with $e$ under the action of the Weyl group on the lattice of weights with a subsequent reduction modulo $p$; see Remark~\ref{rem} for more detail.
When $G$ is an exceptional group and $\tilde{e}$ is rigid and non-special in the sense of Lusztig one can use computations in \cite{Pr14} to determine the central characters of small $\g$-modules with $p$-character $\chi$.\bigskip

\bigskip
\noindent {\bf Acknowledgement.} Part of this work was done in Spring 2018 when the first author was in residence at MSRI (Berkeley). He would like to thank the Institute for the hospitality and support. The work of the second author is funded by the UKRI Future Leaders fellowship program, grant number MR/S032657/1.
%

\section{Preliminaries and recollections}

\subsection{Basic notation}
\label{ss:notationanddefinitions}

In this paper $\Z_+$ will always denote the non-negative integers and $\k$ will denote an algebraically closed field of characteristic $p > 0$. 
Unless otherwise stated an algebraic group will be an affine algebraic group over $\k$. These will be denoted by capital Roman letters 
whilst their Lie algebras will be denoted by the corresponding German scripts. 

The Lie algebra $\g = \Lie(G)$, where $G$ is an algebraic $\k$-group, comes equipped with a canonical $p$-th power operation $\g\ni x \mapsto x^{[p]}\in \g$ equivariant under the adjoint action of $G$ on $\g$.  The {\it nilpotent cone}
$\Nc(\g)$ of $\g$ consists of all $x\in \g$ such that $x^{[p]^r}=0$ for all $r\gg 0$. The adjoint $G$-orbits contained in $\Nc(\g)$ are known as the {\it nilpotent orbits} and they will feature heavily in our discussions.

Given a rational $G$-module $V$ and a vector $v\in V$ we write $G^v$ for the stabiliser of $v$ in $G$. When $x\in \g$ (resp. $\chi\in\g^*$) the Lie algebra $\Lie(G^x)$ (resp. $\Lie(G^\chi)$) is contained in the centraliser $\g^x:=\c_\g(x)$ (resp. in the stabiliser $\g^\chi:=\{y\in \g\,|\,\, \chi([y,\g])=0\}$).


The universal enveloping algebra of $\g$ is denoted by $U(\g)$. 
The the {\it $p$-centre} $Z_p(\g)$ of $U(\g)$ is the unital subalgebra of $U(\g)$ generated by all elements $x^p - x^{[p]}$ with $x\in\g$. It is well-known that $Z_p(\g)$ is contained in the centre of $U(\g)$ and identifies as a $G$-algebra with $\k[(\g^*)^{(1)}]$ where $(\g^*)^{(1)}$ denotes the first Frobenius twist of the coadjoint $G$-module. It follows that the maximal spectrum $Z_p(\g)$ identifies with $(\g^*)^{(1)}$ and, since $\k$ is algebraically closed, the latter can be naturally identified with $\g^*$. Consequently, every $\chi \in \g^*$ leads to a maximal ideal $I_\chi$ of $Z_p(\g)$ and, to make this precise, we have $$I_\chi = \langle x^p - x^{[p]} - \chi(x)^p\mid x\in \g\rangle.$$ The {\it reduced enveloping algebra with $p$-character $\chi$} is the quotient $U_\chi(\g) := U(\g) / I_\chi U(\g)$. By the PBW theorem,  $U(\g)$ is a free module of rank $p^{\dim \g}$ over the $p$-centre $Z_p(\g)$. Therefore, every irreducible representation of $U(\g)$ factors through precisely one of the quotients $U(\g) \twoheadrightarrow U_\chi(\g)$.

\subsection{The standard hypotheses}
\label{ss:standardhypotheses}

We consider a connected reductive algebraic group $G$ over $\k$ satisfying the standard hypotheses:
\begin{enumerate}
\item[(H1)] the derived subgroup of $G$ is simply connected;
\item[(H2)] $p$ is either zero or a good prime for $G$;
\item[(H3)] $\g$ admits a non-degenerate $(\Ad\,G)$-invariant form which we denote $\kappa : \g \times \g \to \k$.
\end{enumerate}
Such groups will be referred as {\it standard} throughout this paper. In this Subsection we recall some useful results on standard reductive groups
which can be found in \cite{SS70}, \cite{Ja98}, \cite{Pr03} and \cite{PS99}. 
Probably the most important result for us is the equality
$\g^x=\Lie(G^x)$ which holds for every $x\in \g$. If $G_1,\ldots, G_s$ are the simple components of the algebraic group $G$, then the Lie algebra $\g$ decomposes into a direct sum of $G$-stable ideals 
\begin{equation}\label{directsum}
\g\,=\,\tilde{\g}_1\oplus\cdots\oplus\tilde{\g}_s\oplus \z,
\end{equation} where
$\z\subseteq \z(\g)$ and $\Lie(G_i)$ is an ideal of codimension $\le 1$ in $\tilde{\g}_i$ for all $i\le s$. More precisely, $\Lie(G_i)=\tilde{\g}_i$ unless $G_i$ is of type ${\rm A}_{rp-1}$ for some $r\in\N$ in which case $\Lie(G_i)\cong \mathfrak{sl}_{rp}(\k)$ and $\tilde{\g}_i\cong \mathfrak{gl}_{rp}(\k)$; see
\cite[2.1]{PS99} for more detail.

 We use the invariant form $\kappa$ to identify the $G$-modules $\g$ and $\g^*$. For any $\chi\in\g^*$ there exists a unique $x\in\g$ such that $\chi=\kappa(x,-)$. The Lie algebras of the Levi subgroups of $G$ are referred to as  the {\it Levi subalgebras} of $\g$. 
If $G$ is standard then so is any Levi subgroup $\uG$ of $G$.
If $x=x_s+x_n$ is the Jordan--Chevalley decomposition of $x$ in the restricted Lie algebra $\g$ then the centraliser $\g^{x_s}$ is Levi subalgebra of $\g$ containing $x$, and $\g^x=(\g^{x_s})^x$.

For the rest of the Subsection we outline how the properties of reductive groups can be used to study nilpotent orbits via reduction modulo $p$.
Let $\mathcal{D}G$ be the derived subgroup of $G$ and 
let $G'_\C$ be a semisimple simply-connected algebraic group over $\C$ whose root system equals that of $G$. It is well-known that $\Nc(\g)\subset\g'$ where
$\g'=\Lie(\mathcal{D}G)$.
The nilpotent part $x_n$ of any $x\in \g$  lies in $\Nc(\g)$.
The number of $G$-orbits in $\Nc(\g)$ is finite (see \cite[\textsection 2]{Ja04} for example) and it follows from \cite{Pr03} that the Dynkin and Bala--Carter classifications of nilpotent orbits work uniformly in the case where $G$ is standard. The uniformity means that the nilpotent orbits of $\g$ are parametrised by the weighted Dynkin diagrams 
of the complex counterpart $\g'_{\mathbb C}=\Lie(G'_\C)$ of $\g'$ in such a way that the nilpotent orbit $\O(\Delta)\subset\Nc(\g)$
with weighted Dynkin diagram $\Delta$ has a nice representative
obtained by base-changing a nilpotent element $e_\Delta$ contained in a Chevalley $\Z$-form $\g'_{\Z}$ of $\g'_{\C}$. The element $e_\Delta\in\g'_\Z$
has several useful properties described in Subsections~\ref{ss:char0nilp} and \ref{ss:RformsQ}. In particular,
it
lies in the adjoint $G'_{\C}$-orbit with weighted Dynkin diagram $\Delta$ and
the dimension of 
that orbit equals $\dim_\k\O(\Delta)$.  The existence of $e_\Delta$ is justified in Subsection~\ref{ss:char0nilp} where we apply \cite[2.6]{Pr03} 
to reduce the general case to the case where $e$ is distinguished. For distinguished nilpotent elements, the existence of $e_\Delta$ is checked   case-by-case by using the tables in \cite{LT2} (for $G$ exceptional) and our description in Subsection~\ref{ss:nilpotentsforclassical} (for $G$ classical).


The results mentioned above enable us to reduce proving Humphreys' conjecture to the case of rigid nilpotent elements in Levi subalgebras of $\g$. In order to solve the problem in those cases 
we introduce, for $e$ rigid, certain $R$-forms of  finite $W$-algebras $U(\g'_\C,e_\Delta)$
and then make use of the results on multiplicity-free primitive ideals
obtained in \cite{Pr14} and \cite{PT14}. Here 
$$R=\Z\big[\textstyle{\frac{1}{p}}\,|\  p\ \mbox{is bad for}\ G\big].$$ In order to construct $R$-forms of rigid finite $W$-algebras we rely in a crucial way
on the structure of the modular Lie algebra $\g^e$ where $e$ is the image of $e_\Delta$ in $\g'\,\cong\, \g'_\Z\otimes_{\Z}\k$.

\subsection{Lusztig--Spaltenstein induction}
\label{ss:LSinduction}
We now recall the theory of induced nilpotent orbits, first developed in \cite{LS79} and generalised to the hypotheses of the current paper in \cite{PS18}. Choose a parabolic subalgebra $\p \subseteq \g$ with Levi factor $\ug \subseteq \p$. Write $\p = \ug \oplus \n$ where $\n$ denotes the nilradical and let $\uO$ be an orbit in $\Nc(\ug)$. The set $\uO + \n$ is contained in $\Nc(\g)$ and so there is a unique nilpotent orbit in $\Nc(\g)$ which intersects this set densely; it is denoted $\Ind_{\ug}^\g(\uO)$ and is referred to as {\it the nilpotent orbit induced from $\uO$}. Crucially, the induced orbit does not depend upon the choice of parabolic subalgebra, only upon the $G$-orbit of the pair $(\ug, \uO)$, which we refer to as {\it the induction data}. The most important feature of this construction for our purposes is that is the following:
\begin{Lemma}
\label{L:inductionproperties}
Let $\ug$ be a Levi subalgebra of $\g$ and let $\uO \subseteq \ug$ be a nilpotent $\uG$-orbit. Then $$\dim \Ind_{\ug}^\g(\uO)
\,=\,\dim \uO+2\dim\n.$$
\end{Lemma}
Nilpotent $G$-orbits which cannot be induced are called {\it rigid} and such orbits play a central role in our proof of the existence of small modules. One of the main results of the first author and David Stewart in \cite{PS18} implies that the classification of rigid orbits does not depend upon the good characteristic.

\begin{Proposition}
	\label{P:reductionrigid}
	Let $G_0$ be an arbitrary simple algebraic $\k$-group satisfying (H1), (H2), (H3) and let $\g_0=\Lie(G_0)$. If $U_\xi(\g_0)$ admits small representations for every nonzero rigid nilpotent $\xi\in \g_0^*$, then 
	$U_\chi(\g)$ admits small representations for every $\chi \in \g^*$.
\end{Proposition}
\begin{proof}
Pick $\chi \in \g^*$. The form $\kappa$ arising from (H3) in Subsection~\ref{ss:standardhypotheses} induces a $G$-equivariant isomorphism $\bar \kappa: \g\to\g^*$ and thus we have a Jordan decomposition $\chi = \chi_{\operatorname{s}} + \chi_{\operatorname{n}}$ into the semisimple and nilpotent parts. The centraliser $\g_{\chi_{\operatorname{s}}}$ is a Levi subalgebra of $\g$ and it is immediate from \cite[Theorem~3.2]{FP80} and  \cite[Proposition~2.5]{Sasha1} that small $U_\chi(\g)$-modules exist if and only if small $U_{\chi_{\operatorname{n}}}(\g^{\chi_{\operatorname{s}}})$-modules exist. The fact that Levi subalgebras of $G$ satisfy the standard hypotheses allows us to reduce the existence of small modules for $U_\chi(\g)$  to the case where $\chi_{\operatorname{s}}$ vanishes on the derived subalgebra of $\g$. In view of \cite[B.9]{JaLA04} we may reduce further to the case where $\chi_{\operatorname{s}} = 0$. Applying \cite[Lemma~2.10]{Ja04} and the Hilbert--Mumford Criterion we may assume from now on that $\chi$ is $G$-unstable and has the form $\chi=\kappa(e,-)$ for some nilpotent element $e\in \g$.

If $U_\chi(\g)$ admits a small module then so too does $U_\psi(\g)$ for all $\psi \in G\cdot \chi$ \cite[A.8]{JaLA04}, and so it suffices to confirm the existence of small modules for a representative $\chi$ of each nilpotent orbit.
Now suppose that $\O\subseteq \g$ is an induced nilpotent orbit. Let $\ug \subseteq \g$ be a Levi subalgebra and $\uO \subseteq \ug$ be a nilpotent orbit such that $\O = \Ind_{\ug}^\g(\uO)$. Choose a parabolic subalgebra $\p \subseteq \g$ which has $\ug$ as a Levi factor, and write $\n$ for the nilradical. Let $\ue \in \uO$ and pick $\ue + n \in (\uO + \n) \cap \O$. We write $\uchi = \kappa(\ue, -)|_{\ug} \in \ug^*$ and $\chi = \kappa(e, -) \in \g^*$. Observe that $\chi|_{\ug} = \uchi$ and $\chi(\n) = 0$.
This leads to a natural embedding of algebras $U_{\uchi}(\ug) \hookrightarrow U_\chi(\g)$, through which we may inflate $U_{\uchi}(\g)$-modules to $U_\chi(\p)$-modules. If $V$ is a small $U_{\uchi}(\ug)$-module then the induced $U_\chi(\g)$-module
$$\widetilde{V}\,:=\,U_\chi(\g)\otimes_{U_\chi(\p)}\,V$$ has dimension $p^{\dim \n+d(\uchi)}$ where 
$d(\uchi)=(\dim\uO)/2$. Applying
Lemma~\ref{L:inductionproperties} we deduce that $\widetilde{V}$ is a small $U_{\chi}(\g)$-module. 
Since any Levi subgroup $\uG$ satisfies the standard hypotheses and $U_0(\ug)$ clearly contains ideals of codimension $1$,
our remarks earlier in the proof show that in order to prove the proposition we may assume without loss of generality that $\ug=\g$ and  
$\uchi=\chi=\kappa(e,-)$ for some nonzero rigid nilpotent element $e\in\g$.  

By (\ref{directsum}), we have that $e=e_1+\cdots+e_s+z$ for some $e_i\in\tilde{\g}_i$ and $z\in\z$. Since all direct summands of (\ref{directsum}) are $G$-invariant and $0\in \overline{G\cdot e}$ it must be that $z=0$ and $0\in\overline{G_i\cdot e_i}$ for all $i$. Since $e$ is rigid in $\g$ and $\{0\}$ is the only rigid unstable orbit in $\gl_n(\k)$ it is straightforward to see that $e_i\ne 0$ only if $\tilde{\g}_i=\g_i$ is not of type $\rm A$. If each such $\g_i$ admits a small module with $p$-character
$\chi_i=\kappa(e_i,-)_{\vert\g_i}$ the the reduced enveloping algebra 
$$U_\chi(\g)\,\cong\,U_\chi(\tilde{\g}_1)\otimes\cdots\otimes U_\chi(\tilde{\g}_s)\otimes U_\chi(\z)$$ admits a module of dimension $p^{(\dim G\cdot \chi)/2}$. This complete the proof.
\end{proof}


Proposition~\ref{P:reductionrigid} reduces the problem of small representations to the case where $G$ is a simple algebraic $\k$-group of type other than $\rm A$, $p$ is a good prime for $G$, and $\chi=\kappa(e,-)$ for some nonzero rigid nilpotent element $e\in \g$.
If $G$ is of type {\rm B}, {\rm C} or {\rm D} then the rigid orbits in $\g$ were first classified by Kempken \cite{K} and Spaltenstein \cite{Sp} in combinatorial terms. Their precise description in terms of partitions is recalled in \cite[Theorem~7]{PT14} and remains valid over algebraically closed fields of characteristic $\ne 2$.
The description of rigid nilpotent orbits in exceptional complex Lie algebras $\g_\C$ can be found in \cite{deGE09} where the authors also show, in the form of tables, how the nilpotent orbits are distributed amongst the sheets of $\g_\C$. It is proved in \cite{PS18} that the results of \cite{deGE09} remain valid in good positive characteristic.

\subsection{Generalised Gelfand--Graev modules and finite $W$-algebras}
\label{ss:finiteWalgebras}
Until the end of this Section we assume that $G$, $\g$ and $\kappa$ are as in Subsection \ref{ss:standardhypotheses}. In order to ease notation we occasionally allow $\k$ to have zero characteristic in which case we set $\k=\C$.
Pick any nilpotent element $e \in \g$ and write $\chi := \kappa(e,-)$ for the associated element of $\g^*$. Thanks to \cite[Theorem~A]{Pr03} there exists a cocharacter $\lambda_e : \k^\times \to G$ which induces a grading $\g = \bigoplus_{i\in \Z}\, \g(i)$ satisfying:
\begin{enumerate}
\setlength{\itemsep}{4pt}
\item[(i)] $e \in \g(2)$;
\item[(ii)] $\g^e \subseteq \bigoplus_{i\ge 0} \g(i)$.
\end{enumerate}
A grading satisfying (i) and (ii) is known as a {\it good grading for $e$}. Such gradings were classified over $\C$ in \cite[Theorem~20]{BG07} and under the standard hypotheses in \cite[\textsection 3]{GT18}. Choose a maximal torus $T \subseteq G$ containing $\lambda_e(\k^\times)$ so that the graded pieces $\g(i)$ are spanned by root vectors for $T$. Property (ii) ensures that $\ad\, e$ induces a bijection $\g(-1) \to \g(1)$. This, in turn, implies that the form $\Psi : \g(-1) \times \g(-1) \to \k$ given by $(x,y)\mapsto \chi([x,y])$ is skew-symmetric and non-degenerate. The following fact was first observed in \cite{BG07} over the complex numbers, and generalised for groups satisfying the standard hypotheses in \cite[\textsection 4.1]{GT18}.
\begin{Lemma}
\label{L:Tstablelagrangian}
There exists a $T$-stable subspace $\g(-1)_0 \subseteq \g(-1)$ which is Lagrangian with respect to $\Psi$.
\end{Lemma}

We fix $\g(-1)_0$ in accordance with Lemma~\ref{L:Tstablelagrangian}, set $\m := \g(-1)_0 \oplus \bigoplus_{i<-1} \g(i)$, and let $P$ be the parabolic subgroup of $G$ with Lie algebra $\bigoplus_{i\le 0}\g(i)$.
Since $\m$ is spanned by $T$-root spaces there exists a connected unipotent algebraic subgroup $M$ of $G$ with $\Lie(M)\,=\,\m$. Moreover, it is explained in \cite[4.2]{GT18} that  $M$ is generated by unipotent root subgroups contained in $R_u(P)$.

Make the notation $\m_\chi := \{x - \chi(x) \mid x\in \m\} \subseteq U(\g)$.
Since $\chi$ vanishes on $\bigoplus_{i\le -3}\,\g(i)$ and $[\m, \m] \subseteq \ker\,\chi$, the preceding remark entails that $\Ad\, M$ preserves the left ideal $U(\g)\m_\chi$ of $U(\g)$ and acts on the quotient  $U(\g)/U(\g)\m_\chi$. From these ingredients we may construct the {\it generalised Gelfand--Graev module} and the {\it finite $W$-algebra}
\begin{eqnarray*}
Q &:=& U(\g)/U(\g)\m_\chi;\\
U(\g,e) &:=& Q^{\Ad\,M}.
\end{eqnarray*}
The subquotient $U(\g,e)$ of $U(\g)$ inherits the structure of an associative $\k$-algebra via $\overline{u} \cdot \overline{v}\, =\, \overline{uv}$ where $\overline{x}$ denotes the coset $x + U(\g) \m_\chi$ of $x\in U(\g,e)$.

When ${\rm char}(\k)=0$ the algebra $U(\g,e)$
quantises the transverse Poisson structure on the Slodowy slice to the adjoint orbit of $e$, which can also be defined by Hamiltonian reduction; see \cite{GG02}. In this case, we also have that $Q^{\Ad\,M} = \,Q^{\ad\,\m}$ which allows us to identify $U(\g,e)$ with $\End_\g(Q)^{\op}$ as associative algebras (via Frobenius reciprocity).

When ${\rm char}(\k)=p>0$ one can define a {\it good transverse slice} to the adjoint orbit of $e$ and show that $U(\g,e)$ quantises its Poisson structure obtained by Hamiltonian reduction, parallel to the characteristic zero case. However, $U(\g,e)$ no longer identifies with $\End_\g(Q)^\op$, but rather with the subalgebra of $M$-invariants; see \cite[Lemma~4.4]{GT18}.

\subsection{The Poincar{\'e}--Brikhoff-Witt theorem}
\label{ss:PBWtheorem}
The PBW theorem for $U(\g,e)$ was first proven for $\k=\C$ in \cite[Theorem~4.6]{Sasha1}, over fields of large positive characteristic in \cite[Lemma~2.1]{Pr10}, and under the standard hypotheses in \cite[Theorem~7.3]{GT18}. Before we record the result we require a little more notation.

The theory of symplectic vector spaces entails that $\g(-1)_0$ admits a Lagrangian complement inside $\g(-1)$, say $\g(-1) = \g(-1)_0 \oplus \g(-1)_1$. We choose a basis $z'_1,\ldots,z'_s$ for $\g(-1)_0$ and a dual basis $z_1,\ldots,z_s$ for $\g(-1)_1$ satisfying
\begin{eqnarray}
\label{e:Lagrangiandualbases}
\Psi(z_i', z_j) = \delta_{i,j}.
\end{eqnarray}
Thanks to property (i) of the good grading on $\g$, the centraliser $\g^e$ is a graded subspace of $\g$. Suppose that $x_1,\ldots,x_r \in \g^e$ is a homogeneous basis for $\g^e$. Thanks to property (ii) of the good grading we can extend this to a homogeneous basis $x_1,\ldots,x_m$ of $\bigoplus_{i\ge 0} \g(i)$. Write $n_i$ for the graded degree of $x_i$. When $(\i, \j) \in \Z_+^m \times \Z_+^s$ we make the notation $x^\i z^\j = x_1^{i_1}\cdots x_m^{i_m} z_1^{j_1}\cdots z_s^{j_s}$ for the PBW monomial, viewed as an element of $Q$. The direct sum decomposition $$\g = \m \oplus \g(-1)_1 \oplus (\bigoplus_{i\ge 0} \g(i))$$
implies that $Q$ is spanned by $\{x^\i z^\j \mid (\i, \j ) \in \Z_+^m \times \Z_+^s\}$, thanks to the classical PBW theorem for $U(\g)$. We define a $\Z_+$-filtration $Q\, = \,\bigcup_{d \in \Z_+} \F_d\, Q$ by placing $\g(d)$ in degree $d+2$ and write $|(\i, \j)|_e$ for the filtration degree of 
the PBW monomial $x^\i z^\j\in Q$. Clearly,
\begin{eqnarray}
\label{e:PBWdegreemonomials}
|(\i, \j)|_e\, =\,\sum_{k=1}^m i_k(n_k + 2) +  \sum_{k=1}^s j_k.
\end{eqnarray}
This filtration on $Q$ induces a non-negative, connected filtration $U(\g,e) = \bigcup_{d\ge 0} \F_d U(\g,e)$ known as the {\it Kazhdan filtration}. 

Given a multi-index ${\bf k}\in\Z_+^d$ we write $|{\bf k}|$ for the total degree of $\bf k$. 
\begin{Lemma}
\label{L:PBWtheorem}
There exist unique elements $\Theta(x_1),\ldots,\Theta(x_r) \in U(\g,e)$ which satisfy
\begin{eqnarray}
\label{e:PBWKazhdanterm}
\Theta(x_k) =\, x_k + \sum_{|(\i, \j)|_e\, \le\, n_k + 2,\ |{\bf i}|+|{\bf j}|\ge 2} \lambda^k_{\i, \j} x^\i z^\j
\end{eqnarray}
where $\lambda^k_{\i,\j} = 0$ whenever ${\bf j}={\bf 0}$  and 
$i_{r+1}=\cdots= i_m=0$.  Furthermore, if $\Theta(x_1),\ldots,\Theta(x_r)$ are elements of $U(\g,e)$ satisfying \eqref{e:PBWKazhdanterm} for certain $\lambda_{\i,\j}^k$ then $U(\g,e)$ admits a basis consisting of the ordered monomials
\begin{eqnarray}
\label{e:PBWbasis}
\{\Theta(x_1)^{i_1}\cdots \Theta(x_r)^{i_r}\,|\,\, (i_1,\ldots,i_r) \in \Z_+^r\}.
\end{eqnarray}
\end{Lemma}
\begin{proof}
When ${\rm char}(\k)=0$ we can take for $\Theta(x_i),\ldots,\Theta(x_r)$ the generators
of $U(\g,e)$ described in \cite[Theorem~4.5(i)]{Sasha1}. 
The uniqueness of such generators follows from the fact that the monomials
$\{\Theta^\i\,|\,\,\i\in\Z_+^r\}$ form a PBW basis of $U(\g,e)$; see \cite[Theorem~4.5(ii)]{Sasha1}. Indeed, if $\Theta(x_i)'$ is another generator satisfying (\ref{e:PBWKazhdanterm}) then $\Theta(x_i)'-\Theta(x_i)\in U(\g,e)$ has no terms supported on
$x_1,\ldots,x_r$. It is easy to see that this contradicts the formula displayed on \cite[p.~27]{Sasha1}.

When ${\rm char}(\k)=p>0$, we start with PBW generators described in  \cite[Theorem~7.3]{GT18}(ii) and then ``improve'' them by subtracting scalar multiples of PBW monomials which correspond to unwanted terms (supported on $x_1,\ldots, x_r$). This process will terminate after finitely many steps.
The uniqueness of $\Theta(x_i)$'s then follows as in the case where ${\rm char}(\k)=0$ since the formula displayed 
on \cite[p.~27]{Sasha1} is still valid in characteristic $p$.
\end{proof}
\subsection{The $p$-central reductions of $U(\g,e)$}
\label{ss:pcentreandSkryabin}
In \cite[Appendix]{Sasha1} Skryabin proved that the category of modules over the finite $W$-algebra $U(\g,e)$ is equivalent to the category of Whittaker $\g$-modules when $\k=\C$. This was inspired by a Morita equivalence proven over fields of positive characteristic by the first-named author and generalised in \cite{Pr10, To17, GT18}. In this paper we only need the characteristic $p$ version of this result, which is enriched by the presence of the $p$-centres of $U(\g)$ and $U(\g,e)$. 


The $p$-centre $Z_p(\g) \subseteq U(\g)$ was described in Subsection \ref{ss:notationanddefinitions}, and we recall that there is a natural $G$-equivariant isomorphism $Z_p(\g) \cong \k[(\g^*)^{(1)}]$. We define $\hZ(\g,e)$ to be the image of the natural map $Z_p(\g) \to Q$. Since $M$ preserves both $Z_p(\g)$ and $U(\g)\m_\chi$ it acts on $\hZ_p(\g,e)$ and the $p$-centre of $U(\g,e)$ is defined to be $Z_p(\g,e)\, :=\, \hZ_p(\g,e)^{\Ad\,M}$.

We pick any graded subspace $\v \subseteq \g$ which is complementary to $[\g,e]$ in $\g$, and put $\Sb_\chi \,:=\, \chi + \bar \kappa(\v)$ where $\bar \kappa$ denotes the isomorphism $\g \to \g^*$ induced by $\kappa$. This variety is known as the {\it good transverse slice} to the coadjoint orbit $\Ad^*(G)\chi$. Since  $\v \subseteq \bigoplus_{i\le 0}\,\g(i)$ is graded  we have a contracting $\k^\times$-action on $\Sb_\chi$ defined by
\begin{eqnarray}
\label{e:contractingaction}
\mu(t) \cdot \eta\, :=\, t^2 (\Ad^*\,\lambda(t))\,\eta\ \quad \qquad \big(\forall\,t\in \k^\times,\ \forall\,\eta \in \Sb_\chi\big) 
\end{eqnarray}
Clearly, $\chi$ is the only fixed point of this action.
There is a natural inclusion $\Sb_\chi^{(1)} \hookrightarrow (\g^*)^{(1)}$ and we shall identify $\Sb_\chi^{(1)}$ with its image in the sequel.

As explained in \cite[\textsection 8]{GT18} the $p$-centre $Z_p(\g,e)$ identifies with the coordinate ring $\k[\Sb_\chi^{(1)}]$ and the finite $W$-algebra $U(\g,e)$ is a free module of rank $p^r$ over $Z_p(\g,e)$ where $r=\dim\g^e$. Hence every $\eta \in \Sb_\chi$ gives rise to a maximal ideal $J_\eta$ of 
$Z_p(\g,e)\,\cong\,\k[\Sb_\chi^{(1)}]$. This leads to a central reduction $$U_\eta(\g,e)\, := \,U(\g,e) / J_\eta U(\g,e)\,\cong\,U(\g,e)\otimes_{Z_p(\g,e)}\,\k_\eta$$ known as the {\it reduced finite $W$-algebra} associated with $\eta$. Each algebra $U_\eta(\g,e)$ has dimension $p^r$ over $\k$. The next result was first proven in \cite[Lemma~2.2]{Pr10}, and then generalised to the hypotheses of the current paper in \cite[\textsection 8 \& \textsection 9]{GT18}.
\begin{Lemma} 
\label{L:Moritatheorem}
Let $\O(\chi)$ denote the coadjoint $G$-orbit of $\chi=\kappa(e,-)$ and $d(\chi)=\frac{1}{2}\dim\O(\chi)$. For every $\eta \in \Sb_\chi$ we have an algebra isomorphism
$$U_\eta(\g) \,\cong\, \Mat_{p^{d(\chi)}} (U_\eta(\g,e)).$$
\end{Lemma}

\subsection{Small representations and rigid nilpotent orbits}
\label{ss:smallrigid}
We retain the setup and notation of the previous Subsection. In particular, we assume that $p={\rm char}(\k)$ is a good prime for $G$. We shall apply Proposition~\ref{P:reductionrigid} to reduce the existence of small modules for reduced enveloping algebras to the existence of one-dimensional representations for finite $W$-algebras associated to rigid nilpotent orbits.
\begin{Proposition}
\label{P:reductionprop}
Suppose that $U(\g_0,e_0)$ admits a one-dimensional representation whenever $G_0$ is a simple algebraic $\k$-group over $\k$ of type other than $\rm A$ and $e_0 \in \g_0$ is a rigid nilpotent element of $\g_0$. Then $U_\chi(\g)$ admits a $p^{d(\chi)}$-dimensional representation for every reductive group $G$ over $\k$ satisfying the standard hypotheses and every $\chi \in \g^*$.
\end{Proposition}
\begin{proof} In view of Proposition~\ref{P:reductionrigid} we need to prove $U_\xi(\g_0)$ admits small representations for every nonzero rigid nilpotent $\xi\in \g_0^*$. To simplify notation we put $\g_0=\g$ and $\xi=\chi=\kappa(e,-)$ where $e$ is a nonzero rigid nilpotent element of $\g$.
The following argument is very similar to \cite[Theorem~2.2]{Pr10}, however we present the details, in the current setting, for the reader's convenience.

By our assumption, the finite $W$-algebra $U(\g,e)$ admits a one-dimensional representation. This representation admits a central character and so, in the notation of Subsection \ref{ss:pcentreandSkryabin}, there exists $\eta \in \Sb_\chi$ and $J_\eta \subseteq Z_p(\g,e)$ such that the representation factors through $U(\g,e) \twoheadrightarrow U_\eta(\g,e)$. Applying Lemma~\ref{L:Moritatheorem} we see that $U_\eta(\g)$ has a representation of dimension $p^{d(\chi)}$ and, thanks to \cite[Theorem~5.4(3)]{PS99} and \cite[Lemma~2.10]{Ja04}, this is equivalent to $U_\eta(\g)$ admitting a two-sided ideal of codimension $p^{2d_\chi}$. It follows from \cite[Lemma~2.3]{PS99} that there is a Zariski closed set $\Xi \subseteq \Sb_\chi$ such that for every $\psi\in \Xi$ the reduced enveloping algebra $U_\psi(\g)$ admits a two-sided ideal of codimension $p^{2d_\chi}$. 
By part~(c) of the proof of \cite[Theorem~2.2]{Pr10}, the set
$\Xi$ is stable under the contracting $\k^\times$-action described in  \eqref{e:contractingaction}. It should be mentioned that in \cite{Pr10} the author works under the assumption that $p \gg 0$, but for this particular argument the standard hypotheses are sufficient. 

Since $\chi$ is contained in every $\k^\times$-stable Zariski closed subset of $\Sb_\chi$ we deduce that the algebra
$U_\chi(\g)$ 
has a homomorphic image of dimension $p^{2d(\chi)}$. But then $U_\chi(\g)$ has a simple module of dimension $\le p^{d(\chi)}$. Applying the main result of \cite{Pr95} completes the proof.
\end{proof}

\section{Nilpotent elements and their centralisers}
In this section we gather together important results regarding the structure of centralisers. To begin with we  consider classical Lie algebras of types {\rm B, C, D} and when doing so, we take $\k$ to be an algebraically closed field such that $\Char(\k) \neq 2$. All vector spaces and algebraic groups are defined over $\k$.

\subsection{Nilpotent elements in symplectic and orthogonal Lie algebras}
\label{ss:nilpotentsforclassical}
Fix a positive integer $N$ and $\epsilon = \pm 1$ with $\epsilon^N=1$. Let $\P(N)$ denote the set of all partitions $\lambda = (\lambda_1,\ldots,\lambda_n)$ of $N$ with $\lambda_1 \ge \lambda_2 \ge \cdots \ge \lambda_n\ge 1$ 
and write $\PeN$ for the subset of $\P(N)$ consisting of all $\lambda$ satisfying the following criteria:
\begin{itemize}
	\item[(i)] if $\epsilon = 1$ then the even parts of $\lambda$ occur with even multiplicity;\smallskip
	\item[(ii)] if $\epsilon = -1$ then the odd parts of $\lambda$ occur with even multiplicity.
\end{itemize}
The importance of $\PeN$ lies in the well-known fact that a nilpotent orbit $\O \subset \gl_N$ with partition $\lambda$ intersects an orthogonal subalgebra $\so_N \subset \gl_N$ if and only if $\lambda \in \P_{1}(N)$, whilst $\O$ intersects a symplectic subalgebra $\sp_N \subset \gl_N$ if and only if $\lambda \in \P_{-1}(N)$ (see \cite[Theorem~5.1.6]{CM93} for example). Furthermore this intersection consists of a unique orbit unless $\epsilon = 1$, all parts of $\lambda$ are even and each occurs with even multiplicity, in which case there are two orbits with partition $\lambda$ labelled $I$ and $I\!I$. Such partitions are called {\it very even}. Note that two orbits with the same very even partition are conjugate under the action of the full orthogonal group \cite[1.12]{Ja04}.

Suppose for a moment that $\k=\C$. We now recall a procedure from \cite{BG07} for choosing an $R$-split (metabolic) bilinear form on $\k^N$ determining a classical Lie subalgebra $\g\subset \gl_N$, as well as a representative of each nilpotent orbit in $\g$ lying in the $R$-span of a Chevalley basis of $\g$. We also explain that the construction in {\it loc.\,cit.} leads to a choice of optimal cocharacter for each representative. 

Let $\{v_i, v_{-i} \mid 1\le i \le \lfloor \frac{N}{2} \rfloor\} \cup \{v_0\}$ be a basis for $\k^N$, where we exclude the element $v_0$ in case $N$ is even. There is a form $(\,\cdot\,,\, \cdot\,) : \k^N \times \k^N \to \k$ satisfying $(u,v) = \epsilon (v,u)$ for every $u,v \in \k^N$ determined uniquely by the following formulae
\begin{eqnarray}
\label{e:bilinearform}
\begin{array}{ccccc}
(v_0, v_i) = 0, & (v_0, v_0) = 2 & (v_i, v_j) = 0 = (v_{-i}, v_{-j}), & \text{ and } & (v_i, v_{-j}) = \delta_{i,j},
\end{array}
\end{eqnarray}
where $1\le i,j \le \lfloor \frac{N}{2}\rfloor$. The connected component of the algebraic subgroup of $\GL_N$ which preserves the form will be denoted $G$, and its Lie algebra by $\g = \Lie(G)$. Since $\Char(\k) \ne 2$ we have
\begin{eqnarray*}
	\g \cong \left\{ \begin{array}{rl} \so_N & \text{ if } \ve = 1, \\ \sp_N & \text{ if } \ve = -1.\end{array}\right.
\end{eqnarray*}
If $J$ denotes the Gram matrix of this linear form then there is an involutive automorphism of the Lie algebra $\gl_N$ defined by $X \mapsto -J^{-1}X^\top J$ which we denote $\sigma$. Here we write $X\mapsto X^\top$ for the transpose of $X\in \gl_N$. The $1$-eigenspace of $\sigma$ is precisely $\g$ and the group $G$ coincides with the identity component of $\GL_N^\sigma$.

We now record a Chevalley basis for $\g$. The standard matrix units in $\gl_N$ are denoted $e_{i,j}$ where $i,j$ vary over $\{0\} \cup \{\pm 1, \pm 2, \cdots, \pm \lfloor \frac{N}{2} \rfloor\}$ and, once again, we exclude the index 0 when $N$ is even. When $\epsilon = 1$ a Chevalley basis is given as follows
\begin{eqnarray}
\label{e:ChevalleybasisOrth}
\begin{array}{cc}
\{e_{i,j} - e_{-j, -i} \mid 1\le i,j \le \lfloor N/2 \rfloor \} \cup \{e_{i, -j} - e_{j, -i}, e_{-j, i} - e_{-i,j} \mid 1\le i< j \le \lfloor N/2 \rfloor \} \\ \ \ \ \ \cup \{2e_{k,0} - e_{0,-k}, e_{0,k} - 2e_{-k,0} \mid 1\le k \le \lfloor N/2\rfloor \}.
\end{array}
\end{eqnarray}
When $\epsilon = -1$ we automatically have $N \in 2\Z$ and a Chevalley basis may be given as follows
\begin{eqnarray}
\label{e:ChevalleybasisSymp}
\begin{array}{cc}
\{e_{i,j} - e_{-j, -i} \mid 1\le i,j \le N/2\} \cup \{e_{i, -j} + e_{j, -i}, e_{-i, j} + e_{-j,i} \mid 1\le i< j \le N/2\} \\ \ \ \ \ \cup \{e_{k,-k}, e_{-k, k} \mid 1\le k \le N/2\}
\end{array}
\end{eqnarray}

Now we fix $\lambda \in \PeN$ and  we call on the notion of {\it the Dynkin pyramid for $\lambda$} defined in Sections 7 and 8 of \cite{BG07}. Informally this is a diagram consisting of $N$ numbered boxes of size $2\times 2$ arranged in the plane, along with some number of crossed boxes. Rather than describing the construction in detail we present here four examples of Dynkin pyramids, and refer the reader to {\it loc. cit.} for the full construction. Here are the Dynkin pyramids associated to the partitions $(5,5,4) \in \P_{-1}(14), (4,3,3,2) \in \P_{-1}(12), (4,4,3,1,1) \in \P_1(13), (5,2,2,1) \in \P_1(10)$:
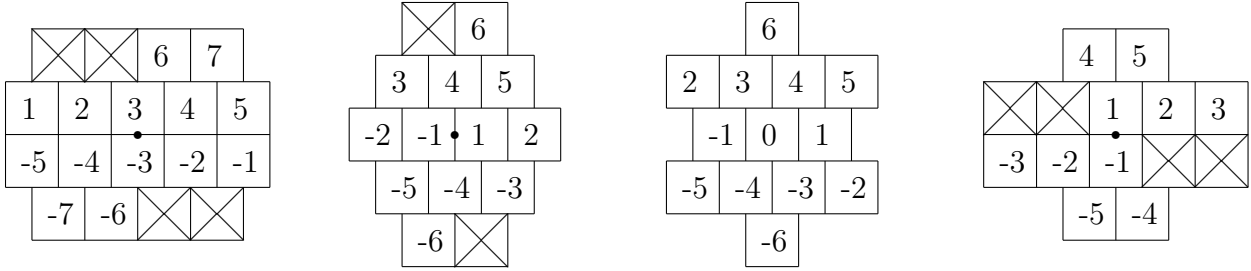
\begin{equation*}
\label{e:apyramid}
\begin{array}{c}
\begin{picture}(480,100)
\put(50,50){\circle*{3}}
\put(10,10){\line(1,0){80}}
\put(0,30){\line(1,0){100}}
\put(0,50){\line(1,0){100}}
\put(0,70){\line(1,0){100}}
\put(10,90){\line(1,0){80}}
\put(0,30){\line(0,1){40}}
\put(20,30){\line(0,1){40}}
\put(40,30){\line(0,1){40}}
\put(60,30){\line(0,1){40}}
\put(80,30){\line(0,1){40}}
\put(100,30){\line(0,1){40}}
\put(10,10){\line(0,1){20}}
\put(30,10){\line(0,1){20}}
\put(50,10){\line(0,1){20}}
\put(70,10){\line(0,1){20}}
\put(90,10){\line(0,1){20}}
\put(10,70){\line(0,1){20}}
\put(30,70){\line(0,1){20}}
\put(50,70){\line(0,1){20}}
\put(70,70){\line(0,1){20}}
\put(90,70){\line(0,1){20}}
\put(10,70){\line(1,1){20}}
\put(30,70){\line(1,1){20}}
\put(30,70){\line(-1,1){20}}
\put(50,70){\line(-1,1){20}}
\put(50,10){\line(1,1){20}}
\put(70,10){\line(1,1){20}}
\put(70,10){\line(-1,1){20}}
\put(90,10){\line(-1,1){20}}
\put(6,56){\hbox{1}}
\put(26,56){\hbox{2}}
\put(46, 56){\hbox{3}}
\put(66, 56){\hbox{4}}
\put(86,56){\hbox{5}}
\put(56, 76){\hbox{6}}
\put(76, 76){\hbox{7}}
\put(86,36){\hbox{-1}}
\put(66,36){\hbox{-2}}
\put(46, 36){\hbox{-3}}
\put(26, 36){\hbox{-4}}
\put(6,36){\hbox{-5}}
\put(36, 16){\hbox{-6}}
\put(16, 16){\hbox{-7}}

\put(170,50){\circle*{3}}
\put(150,0){\line(1,0){40}}
\put(140,20){\line(1,0){60}}
\put(130,40){\line(1,0){80}}
\put(130,60){\line(1,0){80}}
\put(140,80){\line(1,0){60}}
\put(150,100){\line(1,0){40}}
\put(150,0){\line(0,1){20}}
\put(170,0){\line(0,1){20}}
\put(190,0){\line(0,1){20}}
\put(140,20){\line(0,1){20}}
\put(160,20){\line(0,1){20}}
\put(180,20){\line(0,1){20}}
\put(200,20){\line(0,1){20}}
\put(130,40){\line(0,1){20}}
\put(150,40){\line(0,1){20}}
\put(170,40){\line(0,1){20}}
\put(190,40){\line(0,1){20}}
\put(210,40){\line(0,1){20}}
\put(140,60){\line(0,1){20}}
\put(160,60){\line(0,1){20}}
\put(180,60){\line(0,1){20}}
\put(200,60){\line(0,1){20}}
\put(150,80){\line(0,1){20}}
\put(170,80){\line(0,1){20}}
\put(190,80){\line(0,1){20}}
\put(150,80){\line(1,1){20}}
\put(170,80){\line(-1,1){20}}
\put(170,0){\line(1,1){20}}
\put(190,0){\line(-1,1){20}}
\put(176,46){\hbox{1}}
\put(196,46){\hbox{2}}
\put(146, 66){\hbox{3}}
\put(166, 66){\hbox{4}}
\put(186,66){\hbox{5}}
\put(176, 86){\hbox{6}}
\put(156,46){\hbox{-1}}
\put(136,46){\hbox{-2}}
\put(186, 26){\hbox{-3}}
\put(166, 26){\hbox{-4}}
\put(146,26){\hbox{-5}}
\put(156, 6){\hbox{-6}}

\put(280,0){\line(1,0){20}}
\put(250,20){\line(1,0){80}}
\put(250,40){\line(1,0){80}}
\put(250,60){\line(1,0){80}}
\put(250,80){\line(1,0){80}}
\put(280,100){\line(1,0){20}}
\put(280,0){\line(0,1){20}}
\put(300,0){\line(0,1){20}}
\put(250,20){\line(0,1){20}}
\put(270,20){\line(0,1){20}}
\put(290,20){\line(0,1){20}}
\put(310,20){\line(0,1){20}}
\put(330,20){\line(0,1){20}}
\put(260,40){\line(0,1){20}}
\put(280,40){\line(0,1){20}}
\put(300,40){\line(0,1){20}}
\put(320,40){\line(0,1){20}}
\put(250,60){\line(0,1){20}}
\put(270,60){\line(0,1){20}}
\put(290,60){\line(0,1){20}}
\put(310,60){\line(0,1){20}}
\put(330,60){\line(0,1){20}}
\put(280,80){\line(0,1){20}}
\put(300,80){\line(0,1){20}}
\put(286,46){\hbox{0}}
\put(306,46){\hbox{1}}
\put(256,66){\hbox{2}}
\put(276, 66){\hbox{3}}
\put(296, 66){\hbox{4}}
\put(316,66){\hbox{5}}
\put(286, 86){\hbox{6}}
\put(266,46){\hbox{-1}}
\put(316,26){\hbox{-2}}
\put(296, 26){\hbox{-3}}
\put(276, 26){\hbox{-4}}
\put(256,26){\hbox{-5}}
\put(286, 6){\hbox{-6}}

\put(420,50){\circle*{3}}
\put(400,10){\line(1,0){40}}
\put(370,30){\line(1,0){100}}
\put(370,50){\line(1,0){100}}
\put(370,70){\line(1,0){100}}
\put(400,90){\line(1,0){40}}
\put(370,30){\line(0,1){40}}
\put(390,30){\line(0,1){40}}
\put(410,30){\line(0,1){40}}
\put(430,30){\line(0,1){40}}
\put(450,30){\line(0,1){40}}
\put(470,30){\line(0,1){40}}
\put(400,10){\line(0,1){20}}
\put(420,10){\line(0,1){20}}
\put(440,10){\line(0,1){20}}
\put(400,70){\line(0,1){20}}
\put(420,70){\line(0,1){20}}
\put(440,70){\line(0,1){20}}
\put(370,50){\line(1,1){20}}
\put(390,50){\line(-1,1){20}}
\put(390,50){\line(1,1){20}}
\put(410,50){\line(-1,1){20}}
\put(430,30){\line(1,1){20}}
\put(450,30){\line(-1,1){20}}
\put(450,30){\line(1,1){20}}
\put(470,30){\line(-1,1){20}}
\put(416,56){\hbox{1}}
\put(436,56){\hbox{2}}
\put(456, 56){\hbox{3}}
\put(406, 76){\hbox{4}}
\put(426,76){\hbox{5}}
\put(416,36){\hbox{-1}}
\put(396,36){\hbox{-2}}
\put(376, 36){\hbox{-3}}
\put(426, 16){\hbox{-4}}
\put(406,16){\hbox{-5}}

\end{picture}
\end{array}
\end{equation*}
We recall that if the Dynkin pyramid has a zeroth row then it is referred to as a {\it skew row}, and that every row containing crossed boxes is a skew row. In each diagram above a dot ($\bullet$) or a zero is placed at the $(0,0)$ coordinate. The columns and the rows are each numbered by their $x$-coordinate and $y$-coordinate respectively. We write $\col(i)$ for the column of the box containing entry $i$ and similarly we write $\row(i)$ for the row number of $i$. For example, in the Dynkin diagram for $(5,5,4)$ above we have $\row(1) = \cdots = \row(5) = 1, \row(6) = \row(7) = 3, \col(1) = -4, \col(2) = -2, \col(3) = 0, \col(4) = 2, \col(5) = 4, \col(6) = 1, \col(7) = 3$. The rows numbered 3 and $-3$ form a pair of skew rows.

Using this diagram we can introduce a nilpotent element $e\in \g$ as follows. When $\epsilon = 1$ we let $e = \sum \sigma_{i,j} e_{i,j}$ be the sum over all indexes 
$i,j$ such that
\begin{eqnarray*}
	\text{either} & & \col(i) = \col(j) + 2 \text{ and } \row(i) = \row(j);\\
	\text{ or } & & \col(i) = 2, \col(j) = 0 \text{ and } \row(i) = - \row(j) \text{ is a skew
		row in the upper half plane;}\\
	\text{ or } & & \col(i) = 0, \col(j) = -2 \text{ and } \row(i) = - 
	\row(j) \text{ is a skew row in the upper half plane.}
\end{eqnarray*}
Here $\sigma_{i,j} \in \{\pm 1, \pm 2\}$ is the coefficient of $e_{i,j}$ appearing in the basis \eqref{e:ChevalleybasisOrth}. It is not hard to see that this definition gives an element with partition $\lambda$, indeed after checking the example $\lambda = (3,1)$ the general case is clear. To give one example the nilpotent element associated to the partition $(5,2,2,1) \in \P_1(10)$ is
$$e = e_{5,4} - e_{-4, -5} + e_{3,2} - e_{-2, -3} + e_{2,1} - e_{-1, -2}  + e_{1,-2} - e_{2, -1} .$$
In case $\lambda$ is very even, this recipe determines a representative of one of the two orbits associated to the partition $\lambda$. A representative for the other orbit can be obtained by replacing certain $\sigma_{i,j}$ with their negatives; see \cite[p. 26]{BG07}.

When $\epsilon = -1$ we let $e = \sum \sigma_{i,j} e_{i,j}$ where the sum is taken over all $i,j$ such that
\begin{eqnarray*}
	\text{either } & & \col(i) = \col(j) + 2 \text{ and } \row(i) = \row(j);\\
	\text{or } & & \col(i) = 1, \col(j) = -1 \text{ and } \row(i) = - \row(j) \text{ is a skew
		row in the upper half plane.}
\end{eqnarray*}
In this case $\sigma_{i,j} \in \{\pm 1\}$ is precisely the coefficient on $e_{i,j}$ appearing in \eqref{e:ChevalleybasisSymp}. It is straightforward to see that $e$ has partition $\lambda$, indeed there is a permutation matrix which places $e$ in Jordan normal form.

From now on and until the end of this Subsection we assume that $\k$ is an algebraically closed field of characteristic $\ne 2$. The constructions described above provide us, via base change, with nice representatives of nilpotent orbits in $\g=\g_\k$. Furthermore, each Dynkin pyramid determines a cocharacter $\k^\times \to \GL_N$ acting diagonally on our chosen basis for $\k^N$ by setting $t\cdot  e_j = t^{\col(j)} e_j$ for all $-\lfloor N/2\rfloor \le j \le \lfloor N/2\rfloor$. Using the symmetry $\col(-i) = -\col(i)$ of the labelling of the Dynkin pyramid we see that the cocharacter preserves the form defined by \eqref{e:bilinearform}. Thus we obtain the so-called {\it Dynkin cocharacter} $\lambda_e : \k^\times \to G$. This cocharacter induces {\it the Dynkin grading} on $\g$ and $\gl_N$ by the weight space decomposition $\gl_N = \bigoplus_{i\in \Z} \gl_N(i)$ where $\gl_N(i) = \{x\in \gl_N \mid t\cdot x = t^i x \text{ for all } t\in \k^\times\}$. Similarly, $\g = \bigoplus_{i\in \Z} \g(i)$ where $\g(i) = \gl_N(i) \cap \g$. In particular we have $e_{i,j} \in \gl_N(\col(j) - \col(i))$, so that $e\in \g(2)$. 

An important feature of this construction is that the Dynkin grading is good for $e$, which means that $\g^e \subseteq \bigoplus_{i \ge 0} \g(i)$.
The Kempf--Rousseau theory studies the {\it optimal tori for $G$-unstable vectors} which play a definitive role in our work (see \cite[2.2]{Pr03} for a short overview).
\begin{Lemma}
	\label{L:Dynkinoptimal}
	The Dynkin cocharacter is optimal for $e$.
\end{Lemma}
\begin{proof}
	Let $\lambda = (\lambda_1,\dots,\lambda_n) \in \PeN$. The Lie algebra of the Dynkin torus is spanned by an element $h\in\gl_N$ such that $[h,e]=2e$. We claim that $\{e,h\}$ can be completed to an $\sl_2$-triple in $\gl_N$. To see this it will suffice to find elements $\{w_1,\ldots,w_n\} \subseteq \k^N$ such that $\{e^k w_i \mid 1\le i \le n, \ 0\le k < \lambda_i\}$ is a basis and $hw_i = (1-\lambda_i)w_i$. If $\epsilon = -1$ and $\lambda \in \P_{-1}(N)$ then for each non-skew row with $b$ boxes and leftmost entry $i$ the vector space spanned by $\{e^k v_i \mid k=0,\ldots,b-1\}$ is Dynkin graded with $hv_i = (1-b)v_i$. Similarly for each pair of skew rows, each containing $b$ boxes and leftmost entry $i$, the space spanned by $\{e^k v_i \mid k=0,\ldots,2b-1\}$ is Dynkin graded with $hw_i = (1-2b)w_i$. Now $\k^N$ is the inner direct sum of these subspaces, which proves the claim for $\epsilon = -1$. The argument for $\epsilon = 1$ and $\lambda = \P_1(N)$ is quite similar, the only difference is that each pair of skew rows results from a pair of odd parts $\lambda_l> \lambda_m$ occurring with odd multiplicity. If $i$ is the leftmost index of the lower row, and $j$ the leftmost index of the upper row, then we consider the elements
	$$\begin{array}{rcll} 
	v_i & \text{ and } & v_j - v_{-j} & \text{ for } \lambda_m = 1\\
	v_i & \text{ and } & 2v_j - (-1)^{\frac{\lambda_m-1}{2}}v_{-j-\lambda_m + 1} & \text{ for } \lambda_m > 1.
	\end{array}$$
	One can check that in either case these elements have $h$-weight $1-\lambda_l$ and $1-\lambda_m$, that they lie in $\k^N \setminus \im(e)$, and that they generate $\k[e]$-modules of dimension $\lambda_l$ and $\lambda_m$ respectively. Treating the non-skew rows identically to the $\epsilon=-1$ case we have constructed a basis for $\k^N$ with the desired properties, which proves our claim. We write $\{e,h,f\} \subseteq \gl_N$ for our $\sl_2$-triple.
	
	As $\sigma$ preserves $\gl_N(-2)$ and $\Char(\k) \neq 2$ we have that $\gl_N(-2)\,=\,\g(-2)\oplus \q(-2)$ where $\q(-2)=\{x\in\gl_N(-2)\,|\,\,\sigma(x)=-x\}$.
	Writing $f=f_1+f_2$ with $f_1\in \g(-2)$ and $f_2\in \q(-2)$ and using the fact that $h\in \g$ we deduce that $h=[e,f_1]$ and $[e,f_2]=0$.
	As $\gl_N^e(-2)=\{0\}$ this yields $f_2=0$. As a result, $\{e,h,f\}\subset \g$.
	
	When $\Char(\k)=0$ the above argument implies that the Dynkin grading for $e$
	is induced by the adjoint action of a semisimple element coming from an $\sl_2$-triple of $\g$ containing $e$. When ${\rm char}(\k)\ne 2$ the trace form  $(X,Y)\mapsto {\rm tr}(XY)$ on $\g\subset\gl_N$ is non-degenerate. As $\g^e(-2)=\{0\}$ this entails that $[e,\g(0)]=\g(2)$. Consequently, the $C_G(\lambda_e)$-orbit of $e$ is Zariski dense in $\g(2)$.  Since the differential of the Dynkin cocharacter forms part of an $\sl_2$-triple we are in a position to apply \cite[Theorem~2.3]{Pr03}, which shows that the Dynkin cocharacter is optimal for $e$.
\end{proof}

\subsection{The centraliser of an almost rigid nilpotent element} 
In this Subsection we shall temporarily switch to more traditional representatives
of nilpotent orbits in $\sp_N$ and $\so_N$ described by Springer and Steinberg in \cite[Ch.~IV, \textsection 2]{SS70}. 
This is harmless when one works over an algebraically closed field of characteristic $\ne 2$ as 
in this case one can pass from a Springer--Steinberg representative to a suitable representative $e$ from Subsection~\ref{ss:nilpotentsforclassical}
by means of an element in ${\rm SO}(N)$ or $\Sp(N)$ which commutes with the Dynkin cocharacter
$\lambda_e$. This will be explained in more detail in the proof of
Lemma~\ref{L:basis-g-e}. 

Fix $\lambda = (\lambda_1 \ge \cdots \ge \lambda_n)$ and keep fixed the choice of nilpotent element $e\in \g$ from the previous section. The following result shows that $\g^e$ admits a nice basis, which can be described uniformly for $\epsilon = \pm 1$.
\begin{Lemma}\label{L:basis-g-e}
	\label{L:spanningsetlemma}
	There exists an involution $i \mapsto i'$ on the set $\{1,\ldots,n\}$ and a spanning set for $\g^e$
	\begin{eqnarray}
	\label{e:spanningforge}
	\begin{array}{c}
	\{\zeta_i^{j,s} \mid\,  1\leq i,j \leq n,\ 0 \leq s < \min(\lambda_i, \lambda_j)\}\vspace{8pt}\\
	\end{array}
	\end{eqnarray}
	such that
	\begin{enumerate}
		\setlength{\itemsep}{4pt}
		\item $\lambda_i = \lambda_{i'}$ for all $i = 1,\ldots,n$;\smallskip
		\item $i = i'$ if and only if $\epsilon(-1)^{\lambda_i} = 1$;\smallskip
		\item $i' \in \{i-1, i, i+1\}$.
	\end{enumerate}
	and the following properties hold:
	\begin{enumerate}
		\setlength{\itemsep}{4pt}
		\item[(i)] The linear relations amongst \eqref{e:spanningforge} are of the form $\zeta_i^{j,s} = \ve_{i,j,s}\zeta_{j'}^{i', s}$ where 
		\begin{eqnarray}
		\begin{array}{c}
		\label{e:defineve}
		\ve_{i,j,s} := \varpi_{i\leq i'} \varpi_{j\le j'} (-1)^{\lambda_j-s};\vspace{8pt}\\
		\varpi_{i\le i'} := \left\{\begin{array}{cl} 1 & \text{ if } i \leq i'; \\ -1 & \text{ if } i > i'.\end{array}\right.
		\end{array}
		\end{eqnarray}
		
		\item[(ii)] $[\zeta_i^{j,s}, \zeta_k^{l,r}] = \delta_{il} \zeta_k^{j,r+s - (\lambda_i - 1)} -
		\delta_{jk} \zeta_i^{l,r+s - (\lambda_j - 1)}+ \varepsilon_{k,l,r}\big(\delta_{k,i'}
		\zeta_{l'}^{j,r+s - (\lambda_i - 1)} - \delta_{j,l'} \zeta_i^{k',r+s - (\lambda_j - 1)}\big)$.
		
		\item[(iii)] $\zeta_i^{j,s} \in \g^e(\lambda_i + \lambda_j - 2s - 2)$.
	\end{enumerate}
	In (i), (ii), (iii) the indexes $i,j,s,k,l,r$ vary in the ranges permitted by \eqref{e:spanningforge}, and in (ii) we adopt the convention $\zeta_i^{j,s} = 0$ for $s < 0$.
\end{Lemma}
\begin{proof}
	The existence of an involution satisfying (1), (2), (3) is an immediate consequence of the parity conditions on $\lambda = (\lambda_1,\ldots,\lambda_n)$ described at the start of Subsection~\ref{ss:nilpotentsforclassical}. 
	
	It is well-known that all optimal cocharacters for $e$ with respect to which $e$ has weight $2$  are $C_G(e)$-conjugate. We sketch the argument here for the reader's convenience, drawing on the terminology used in \cite{Pr03}. Let $\gamma$ and $\gamma'$ be two such cocharacters, so that $\gamma(\k^\times), \gamma'(\k^\times)$ normalise $\k e$, hence lie in $\gamma(\k^\times) G^e$. As all maximal tori of $G^e$ are conjugate there is $g\in G^e$ such that $\gamma(\k^\times)$ and $g\gamma'(\k^\times)g^{-1}$ belong to the same maximal torus of $G^e$, which in turn is contained in a maximal torus of the optimal parabolic subgroup of $G$, say $T$.
	Then the optimal cocharacters $\gamma$ and $g\gamma' g^{-1}$ lie in $X_*(T)$. By one of the main result of the Kempf--Rousseau theory \cite[Theorem 2.1(iii)]{Pr03} we know that  $X_*(T)$ contains a unique primitive, optimal element.
	Hence there is a nonzero $d\in \Q$ such that $d\gamma=g\gamma' g^{-1}$. As $e$ has weight $2d$ for $d\gamma'$ 
	we get $d=1$, which shows that $\gamma$ and $\gamma'$ are $G^e$-conjugate.
	
	Using the fact that all non-degenerate symmetric and skew-symmetric forms on $\k^N$ are $\GL_N$-conjugate, along with the classification of nilpotent $G$-orbits and the remarks of the previous paragraph, we see that it suffices to prove the lemma for any choice of $G=G_{(\,\cdot\,,\,\cdot\,)}$ where ${(\,\cdot\,,\,\cdot\,)}$ is a non-degenerate form on $\k^N$ satisfying $(u,v) = \epsilon (v,u)$, any choice of nilpotent $e\in \g_{(\,\cdot\,,\,\cdot\,)}$ with partition $\lambda$, and any choice of $\gamma\in X_*(G)$ optimal for $e$ with the property that $e$ has weight $2$ with respect to $\gamma$.
	
	If we choose $\g = \g_{(\,\cdot\,,\,\cdot\,)}$ and $e$ in accordance with \cite[Section 2]{PT14} then (i) and (ii) follow from Lemmas 2 and 5 of {\it loc.\,cit.} The element $e$ is in Jordan normal form, and $\k^N$ admits a basis of the form $\{e^k w_i \mid 0 \le k < \lambda_i, \ 1\le i \le n\}$. We can define a cocharacter $\gamma : \k^\times \to \GL_N$ by setting $t\cdot e^k w_i = t^{1-\lambda_i + 2k} e^k w_i$. This cocharacter factors through $G_{(\,\cdot\,,\,\cdot\,)}$. It is straightforward to complete $\{{\rm d}_1\gamma, e\}$ to an $\sl_2$-triple, and so we may apply precisely the same argument as Lemma~\ref{L:Dynkinoptimal} to deduce that $\gamma$ is optimal for $e$. Now it remains to observe that the basis \eqref{e:spanningforge} of $\g^e$ is graded by the formula in part (iii) of the current Lemma with respect to this cocharacter. This follows from a short calculation.
\end{proof}

We write $\g = \bigoplus_{i\in \Z} \g(i)$ for the Dynkin grading coming from the Dynkin cocharacter for $e$. We say that $\lambda$ is {\it almost rigid} if $\lambda_i - \lambda_{i+1} \le 1$ for all $i\le n$, where $\lambda_{n+1}=0$ by convention, and we say that $e\in \g$ is {\it almost rigid} if the associated partition is so. It is well-known that all rigid nilpotent elements of $\g$ are almost rigid; see \cite{PT14}, for example.
The following is the first main result of this section. It should be mentioned that closely related results were obtained by Yakimova in \cite{Ya10} over the complex numbers. However, some of Yakimova's arguments rely on the fact that in characteristic zero $\sl_N$ 
does not contain nonzero scalar matrices which fails
in characteristic $p$ when $p\mid N$.
\begin{Proposition}
	\label{P:arprop}
	If $e$ is almost rigid then the Lie algebra
	$\g^e$ is generated by $\g^e(0)$ and $\g^e(1)$.
\end{Proposition}
\begin{proof}
	
	Thanks to \cite[Theorem~6]{PT14} we know that a complement to $[\g^e, \g^e]$ in $\g^e$ is given by
	$$\{\zeta_i^{i+1, \lambda_{i+1} - 1} \mid i = i', i+1 = (i+1)', \lambda_{i-1} \ne \lambda_i \ge \lambda_{i+1} \ne \lambda_{i+2}\}$$
	and, since $\lambda$ is almost rigid, these elements all lie in Dynkin degree zero by Lemma~\ref{L:spanningsetlemma}(iii). Therefore, to prove the Proposition, we shall demonstrate that $[\g^e(1), \g^e(r-1)] = \g^e(r)$ for all $r > 1$.

	For the rest of the proof we fix $r > 1$. In order to prove that $[\g^e(1), \g^e(r-1)] = \g^e(r)$, we choose $i,j,s$ in the ranges prescribed by \eqref{e:spanningforge} satisfying $\lambda_i + \lambda_j - 2s - 2 = r$, and show that $\zeta_i^{j,s}$ lies in $[\g^e(1), \g^e(r-1)]$. Throughout the proof we often use the fact that $i \neq j'$ whenever $\lambda_i \neq \lambda_j$, which is an easy consequence of property (2) in Lemma~\ref{L:spanningsetlemma}.
	
	First of all suppose that $|\lambda_i - \lambda_j| \ge 2$. Thanks to Lemma~\ref{L:spanningsetlemma}(i), we may suppose that $i < j$. Using the fact that $\lambda$ is almost rigid we know that there exists a $k$ with $i < k < j$ and $\lambda_i = \lambda_k+1 > \lambda_j$. Using \eqref{e:defineve}  and examining the Dynkin degrees in Lemma~\ref{L:spanningsetlemma}(iii) we have that $$\zeta_i^{j,s}\, = \,[\zeta_k^{j, s}, \zeta_i^{k, \lambda_k-1}] \in\, [\g^e(1), \g^e(r-1)]\qquad\quad\big(\forall
	\ \zeta_i^{j, s}\in \g^e(r)\big).$$
	Here we also use Lemma~\ref{L:spanningsetlemma}(ii)  keeping in mind that
	$\delta_{i,k'}=\delta_{j,k'}=0$ since $\lambda_k=\lambda_{k'}\neq \lambda_i$ and $\lambda_{k'}=\lambda_i-1>\lambda_j$.
	
	Next we suppose that $|\lambda_i - \lambda_j| = 1$; similar to the previous case we may (and shall) assume that $i < j$. Since $\lambda_i$ and $\lambda_j$ have opposite parity we either have $j \neq j'$ or $i \neq i'$. Since the arguments are similar in these two cases we only present the details for the former. The Dynkin degree of $\zeta_i^{j,s}$ is $r > 1$ and so we must have $s\neq \lambda_j -1$. The assumption $j \neq j'$ ensures that $\zeta_j^{j,s} \neq 0$ has a non-zero Dynkin degree, whilst $\zeta_i^{j,\lambda_j-1} \in \g^e(1)$. Now an application of Lemma~\ref{L:spanningsetlemma}(ii) confirms that $\zeta_i^{j,s} = [\zeta_j^{j,s}, \zeta_i^{j,\lambda_j-1}]$ is an element of $[\g^e(1), \g^e(r-1)]$ (one should keep in mind that in the present case $\delta_{j,i}=\delta_{i,j'}=\delta_{j,j'}=0$).
	
	It remains to consider the case where $\lambda_i = \lambda_j$ and $\lambda_i + \lambda_j - 2s - 2 = 2(\lambda_i - 1 - s)  = r > 1$. We are going to use induction on $\lambda_i$ to show that $\zeta_i^{j,s}$ lies in $[\g^e(1), \g^e(r-1)]$.  So we fix $i$ and assume, for all $k, l, t$ satisfying $\zeta_k^{l,t} \in \g^e(r)$, that $\zeta_k^{l, t}\in [\g^e(1), \g^e(r-1)]$ whenever $\lambda_k < \lambda_i$. Since $\lambda$ is almost rigid and we require $2(\lambda_i - 1 - s)  = r > 1$, there is $k$ with $\lambda_k=1$, but our induction starts with the case where $\lambda_i = \lambda_j = 2$ and $s = 0$. Using Lemma~\ref{L:spanningsetlemma}(ii) once again we get $[\zeta_k^{j, 0}, \zeta_i^{k,0}] = \zeta_i^{j,0} - \delta_{i,j} \zeta_{k}^{k,-1} = \zeta_{i}^{j,0}$, which deals with the induction base. The inductive step is very similar: we fix $i,j,s$ satisfying the conditions listed at the beginning of this paragraph and we choose $k$ such that $\lambda_k = \lambda_i - 1$. The conditions on $i,j,s$ ensure that $s < \min(\lambda_i, \lambda_k) = \lambda_k$ and so $\zeta_i^{k,s}$ occurs in \eqref{e:spanningforge}. Now we have $[\zeta_k^{j,\lambda_k-1}, \zeta_i^{k,s}] = \zeta_i^{j,s} - \delta_{i,j} \zeta_k^{k,s-1}$. By the inductive hypothesis we have $\zeta_k^{k,s-1} \in [\g^e(1), \g^e(r-1)]$ and it follows that $\zeta_i^{j,s} \in [\g^e(1), \g^e(r-1)]$ as well. This completes the proof of the proposition.
\end{proof}

\begin{Remark}
When $\g^e\,=\,\Lie(G^e)$,  the argument used in the proof of Proposition~\ref{P:arprop} also works in characteristic $2$. This situation arises when
the Hesselink stratum of $\mathcal{N}(\g)$ containing $e$ coincides with the adjoint orbit $\O(e)$; see \cite{Cl-Pr} for more detail on the Hesselink stratification of nilpotent cones.

\end{Remark}
\subsection{Nilpotent elements of $\g$ and their  counterparts over $\Z$}\label{ss:char0nilp}
Let $G_\Z$ be a split reductive $\Z$-group scheme with root datum
$(X(T_\Z), \Phi, X_*(T_\Z), \Phi^\vee)$ where $T_\Z$ is a fixed maximal 
torus of $G_\Z$ split over $\Z$.
We assume that $\Phi$ is an irreducible root system of type other than $\rm A$ and $X_*(T_\Z)$ is spanned over $\Z$ by the coroots of $\Phi$. 
We let $\Pi$ be a basis of simple roots in $\Phi$ and assume henceforth that
the lattice of coweights $X_*(T_\Z)$ of $T_\Z$ is a free $\Z$-module with basis $\{\alpha^\vee\,|\,\,\alpha\in\Pi\}$. Then 
the lattice of weights $X(T_\Z)$ of $T_\Z$ is spanned over $\Z$ by the fundamental weights corresponding to $\Pi$. Let $\Phi_+(\Pi)$ and $\Phi_-(\Pi)$
be the sets of positive and negative roots of $\Phi$, respectively.

Given a commutative associative ring $A$ with $1$ we denote by $G_A$ the group of $A$-points of $G_\Z$. The Lie algebra $\g_\Z$ of the group scheme $G_\Z$ is a lattice in the complex simple Lie algebra $\g_\C=\Lie(G_\C)$
spanned over $\Z$ by a Chevalley basis $\{h_\alpha\,|\,\,\alpha\in\Pi\}\cup\{e_\alpha\,|\,\,\alpha\in\Phi\}$.
Here $h_\alpha={\rm d}_e\alpha^\vee$. Note that $\Z e_\alpha=\,\Lie(U_{\Z,\alpha})$, where $U_{\Z,\alpha}$ is the root subscheme of $G_\Z$ corresponding to $\alpha\in\Phi$, and $\Lie(T_\Z)=\bigoplus_{\alpha\in\Pi}\,\Z\, h_\alpha$.
If $K$ is a field then $\Lie(G_K)\,\cong\,\g_\Z\otimes_\Z K$ as Lie algebras over $K$. 
	
In this Subsection we always assume that ${\rm char}(\k)$ is  good for the root system $\Phi$.
Our next goal is to assign to any nilpotent $G_\k$-orbit $\O_\k$ in $\g_\k$ a nice nilpotent element $e\in \g_\Z$ whose image $e\otimes_\Z 1$ in $\g_\k\,\cong\,\g_\k\otimes_\Z\k$ lies in $\O_\k$. 
We shall rely on basic notions of the Bala--Carter theory and use the notation introduced in \cite[2.6]{Pr03}. 

Let $\Pi=\{\alpha_1,\ldots,\alpha_l\}$. Given a subset $I$ of $\{1,\ldots,l\}$
we set $\Pi_I:=\{\alpha_i\,|\,\,i\in I\}$ and denote by $L_{I, \Z}$ the standard Levi subgroup of $G_\Z$ corresponding to $I$. Recall that the root system $\Phi_I$ of $L_{I, \Z}$ with respect to $T_\Z\subset L_{I,\Z}$ consists of all $\beta\in\Phi$ such that $\beta=\sum_{i\in I}n_i\alpha_i$ where $n_i\in \Z$. The Lie algebra of $L_{I,\Z}$ is spanned over $\Z$ by $\Lie(T_\Z)$ and all $e_\alpha$ with $\alpha\in \Phi_I$. Given a subset $J$ of $I$ we denote by $P_{I,J,\,\Z}$ the standard parabolic subgroup scheme of 
$L_{I,\Z}$ associated with $J$. The Lie ring $\p_{I,J,\,\Z}\,=\,\Lie(P_{I,J,\,\Z})$ is spanned over $\Z$ by $\Lie(T_{\Z})$ and all $e_\beta$ with $\beta=\sum_{i\in I}n_i\alpha_i\in \Phi_I$ such that $n_i\ge 0$ for all $i\in I\setminus J$. Note that $L_{J,\Z}$ is a Levi subgroup scheme of $P_{I,J,\,\Z}$ and $\p_{I,J,\,\Z}$ has a natural $\Z$-grading
$\p_{I,J,\,\Z}\,=\,\bigoplus_{k\ge 0}\,\p_{I,J,\,\Z}(k)$ such that
$\p_{I,J,\,\Z}(0)=\Lie(L_{J,\Z})$ and $\p_{I,J,\,\Z}(k)$ with $k>0$  is spanned over $\Z$ by all $e_\beta$ with $\beta=\sum_{i\in I}n_i\alpha_i$ and
$\sum_{i\in I\setminus J}n_i=k$. 

In what follows we shall be particularly interested in the graded components $\p_{I,J,\,\Z}(2)$ associated with {\it some} pairs $(I,J)$.
Let $W=W(\Phi)$ be the Weyl group of $G_\Z$ with respect to the split maximal torus $T_\Z$. Following \cite[2.6]{Pr03} we denote by $\P(\Pi)$ the set of all
pairs $(I,J)$ with  $J\subseteq I\subseteq \{1,\ldots,l\}$ such that
$P_{I,J,\,\C}$ is a {\it distinguished} parabolic subgroup of $L_{I,\C}$. Two pairs $(I,J)$ and $(I',J')$ ar said to be equivalent if there is $w\in W$ such that $w(\Pi_I)=\Pi_{I'}$ and $w(\Pi_J)=\Pi_{J'}$ and we write $[P(\Pi)]$
for the set of all equivalence classes. The main result of the Bala-Carter theory states that the set $[P(\Pi)]$ parametrises the nilpotent orbits in $\g_\C$ and it is proved in \cite[Theorem~2.6]{Pr03} that this continues to hold in good characteristic. We denote by $\O_{\C}(I,J)$ and $\O_\k(I,J)$
the nilpotent orbits of $\g_\C$ and $\g_\k$ associated with $(I,J)\in\P(\Pi)$.

In order to identify nice representatives of nilpotent orbits in $\g_\k$ and their counterparts in $\g_\Z$, we must study the interplay between the Bala-Carter theory, the Dynkin theory and the Kempf--Rousseau theory.

As explained in \cite[2.6]{Pr03}, for any $(I,J)\in\P(\Pi)$ there exists a unique element $h_{I,J}\,=\,\sum_{i\in I}a_i h_{\alpha_i} \in \Lie(T_\Z)$ such that $\alpha_k(h_{I,J})=0$ for all $k\in J$ and $\alpha_k(h_{I,J})=2$ for all $k\in I\setminus J$. We set $\lambda_{I,J}\,:=\,\sum_{i\in I}a_i\alpha_i^\vee$, an element of $X_*(T_\Z)$ (it is important here that $a_i\in \Z$ for all $i\in I$). There exists $w\in W$ such that $w(\lambda_{I,J})$ lies in the dual Weyl chamber associated with $\Pi$.
Placing the nonnegative integers $\alpha(w(\lambda_{I,J}))$ with $\alpha\in\Pi$ atop of the corresponding vertices of the Dynkin graph of $\Pi$ one obtains the weighted Dynkin diagram,
$\Delta_{I,J}$, of $\O_{\C}(I,J)$ shared by $\O_\k(I,J)$. 
The adjoint action of 
$\lambda_{I,J}({\mathbb G}_{\rm m})$ on $\g_\Z$ gives $\g_A=\g_\Z\otimes_\Z A$ a $\Z$-graded Lie ring structure and we denote by $\g_A(k)$ the $k$-th graded component of $\g_A$. It is immediate from the definition of $\lambda_{I,J}$ that $\p_{I,J,\,\Z}(k)\,=\, \p_{I,J,\,\Z}\cap \g(k)$ for all $k\in \Z$.

The remaining part of this section contains several results which resemble recent results of McNinch obtained in 
\cite[Section~3]{McN}. It should be stressed at this point that McNinch works locally, that is in the setting of a reductive group scheme defined 
over a discrete valuation ring whose residue field has characteristic $p>0$ whereas we need global statements which would apply uniformly over the whole range of residue fields.

\begin{Proposition}\label{e-Delta}
Let $\k$ be an algebraically closed field whose characteristic is good for the root system $\Phi$ and let $\lambda_{I,J}\colon\,{\mathbb G}_{\rm m}\to T_\Z$
be the cocharacter associated with a pair $(I,J)\in\P(\Pi)$. 
\begin{enumerate}
	\item [(i)\,] There exists a principal Zariski open subscheme  $\p_{I,J,\,\Z}(2)_{\rm reg}$ of $\p_{I,J,\,\Z}(2)$ such that  $\p_{I,J,\,\k}(2)_{\rm reg}\,=\,\p_{I,J,\,\k}(2)\cap \O_\k(I,J)$ for any field $\k$ as above. 
	
	\smallskip
	
	\item[(ii)\,]  The cocharacter $\lambda_{I,J}$ is optimal in the sense of the Kempf--Rousseau theory for any element of $\p_{I,J,\,\k}(2)_{\rm reg}$ and any field $\k$ as above.
	
	\smallskip
	
	\item[(iii)\,] There exits $e_{I,J}\in \p_{I,J,\,\Z}(2)_{\rm reg}\cap\O_\C(I,J)$ whose image $e_{I,J}\otimes_\Z 1$ in 
	$\g_\k=\g_\Z\otimes_\Z \k$ lies in $\p_{I,J,\,\k}(2)_{\rm reg}$ 
	for any field $\k$ as above.
	
	\smallskip
	
	\item [(iv)\,] The reductive Lie algebra $\g_\Q^{e_{I,J}}\cap\g_\Q(0)$ is split over $\Q$. 	
\end{enumerate}
\end{Proposition}
\begin{proof}
Parts~(i) follows from \cite[Theorem~2.6(ii)]{Pr03} and the first part of the proof of \cite[Theorem~2.3]{Pr03}. A slightly more general result valid in arbitrary characteristic can be found in \cite[Remark~7.3(3)]{Cl-Pr}.
More precisely, it is explained in {\it loc.\,cit} that there exists a nonzero homogeneous $\varphi\in S(\g_\Z(2)^*)$ semi-invariant under the adjoint action of the Levi subgroup $C_{G_\C}(\lambda_{I,J})$ of $G_\C$ and such $$\p_{I,J,\,\k}(2)\cap \O_\k(I,J)\,=\,\{x\in  \p_{I,J,\,\k}(2)\,|\,\, \bar{\varphi}(x)\ne 0\}$$ where $\bar{\varphi}$ is the image of $\varphi$ in $S(\g_\k(2)^*)$. Hence we can define
$\p_{I,J,\,\Z}(2)_{\rm reg}\,:=\,\p_{I,J,\,\Z}(2)_\varphi$, a principal open subscheme of $\p_{I,J,\,\Z}(2)$. Part~(ii) is now an immediate consequence of \cite[Theorem~2.6(iii) and Theorem~2.3(i)]{Pr03}.

In order to prove (iii) we need to find an element $e_{I,J}\in \p_{I,J,\,\Z}(2)$ such that $\varphi(e_{I,J})\not\equiv 0 \pmod{p}$ for any good prime $p$. Note that $\bar{\varphi}$ does not vanish on $\p_{I,J,\,\k}(2)$ because $\p_{I,J,\,\k}(2)\cap \O_\k(I,J)\ne \emptyset$.
Therefore, it suffices to ensure that $\bar{e}_{I,J}$, the image of $e_{I,J}$ in $\p_{I,J,\,\k}(2)$, has the property that 
\begin{equation}\label{surj}
[\bar{e}_{I,J},\p_{I,J,\,\k}(0)]\,=\,\p_{I,J,\,\k}(2)	
\end{equation}	
holds in arbitrary good characteristic. Indeed, as $\p_{I,J,\,\k}(0)\,=\,\Lie(L_{J,\k})$, it follows from (\ref{surj})
that the adjoint $L_{J,\k}$-orbit of $\bar{e}_{I,J}$ is dense in $\p_{I,J,\,\k}(2)$, forcing $\bar{\varphi}(\bar{e}_{I,J})\ne 0$.

Our assumptions on $G_\Z$ imply that the Levi subgroup $L_{J,\k}$ satisfies the standard hypotheses  (cf. Subsection~\ref{ss:standardhypotheses}). This means that in our search of an element $e:=e_{I,J}$ satisfying (\ref{surj}) we may assume without loss of generality that 
$I=\{1,\ldots,l\}$ and 
$L_{J,\Z}=G_\Z$. As $\g_\k$ admits a non-degenerate symmetric bilinear form, we seek an element $e\in \g_\Z(2)$ whose image $e_\k\in\g_{\k}(2)$ has the property that 
$\g_\k^{e_\k}\cap\g_\k(-2)\,=\,\{0\}$.

For $\Phi$ exceptional, suitable representatives in 
$\O(I,J)\cap \g_\Z(2)$ can found in the tables of \cite{LT2} (one only needs to look at the pages devoted to the distinguished nilpotent $G_\k$-orbits). If $\Phi$ is of type $\rm B$, $\rm C$ or $\rm D$, one can choose  $e=\sum\sigma_{i,j}e_{i,j}$ as described in Subsection~\ref{ss:nilpotentsforclassical}. Again, one only needs to look at the cases where $e$ is distinguished (which happens when the partition $\lambda=(\lambda_1\ge\lambda_2\ge\cdots\ge \lambda_n)$  associated with $e$ has pairwise  distinct parts and $i=i'$ for all $i\le n$).
The Chevalley bases (\ref{e:ChevalleybasisOrth}) and (\ref{e:ChevalleybasisSymp}) provide us with a natural choice of a maximal torus $T_\Z\subset G_\Z$. Since the Dynkin cocharacter $\lambda_e\in X_*(T_\Z)$ is optimal for $e$ by Lemma~\ref{L:Dynkinoptimal} and $e$ has weight $2$ with respect to $\lambda_e$ we can find a basis of simple roots $\Pi=\{\alpha_1,\ldots,\alpha_l\}$ of the root system $\Phi=\Phi(G_\Z,T_\Z)$ and a subset $J$ of $\{1,\ldots,l\}$ such that
$\lambda_e=\lambda_{\{1,\ldots,l\},J}$ and $e=e_{\{1,\ldots,l\},J}$

Part~(iv) follows from the fact that any $e=e_{I,J}\in\p_{I,J,\,\Q}(2)$ is a distinguished nilpotent element of the derived subalgebra of $\Lie(L_{I,\Q})$. Let $\t_{e,\Q}:=\{h\in \Lie(T_\Q)\,|\,\,\alpha_i(h)=0\ \, \text {for all}\ i\in I\}$. Since the Levi subalgebra $$\Lie(L_{I,\Q})\,=\,\t_{e,\Q}\oplus [\Lie(L_{I,\Q}),\Lie(L_{I,\Q})]$$ coincides with the centraliser of $\t_{e,\Q}$ in $\g_\Q$ and the centraliser of $e$ in the derived subalgebra of $\Lie(L_{I,\Q})$ consists of nilpotent elements, we see that $\t_{e,\Q}$ is a $\Q$-split maximal toral subalgebra of $\g_\Q^e$. Since it is well-known that the Lie algebra $\g_\Q^e\cap \g_\Q(0)$ is reductive and the maximal toral subalgebra $\t_{e,\Q}$ of $\g_\Q^e$ is contained in $\g_\Q^e\cap \g_\Q(0)$, this completes the proof of the Proposition.
\end{proof}
\subsection{Properties of the centraliser of $e_{I,J}$ in $\g_R$} \label{ss:geR}
Recall the ring $R$ from Subsection~\ref{ss:standardhypotheses}. In this Subsection we fix $(I,J)\in\P(\Pi)$
and denote by $e_\k$ the image of $e:=e_{I,J}$ in $\g_\k=\g_R\otimes_R\k$.
Set $\t_A\,:\,=\Lie(T_\Z)\otimes_\Z A$. Our choice of $R$ implies that the inverse of the Cartan matrix associated with $\Pi$ has entries in $R$. As a consequence, there exist $t_1,\ldots, t_l\in \t_R$ such that $\alpha_i(t_j)=\delta_{i,j}$ for all $1\le i,j\le l$. 
We denote by $\t_{e,R}$ the $R$-span of all $t_i$ with $i\not\in I$
and identify the sublagebra $\t_{e,\Q}$  
introduced in proof of Proposition~\ref{e-Delta}(iv)
with $\t_{e,R}\otimes_R\Q$. It is clear (and will be important in what follows) that the Lie ring
$\g_R$ decomposes into a direct sum 
\begin{equation}
\label{Rweights}
\g_R\,=\,\textstyle{\bigoplus}_{\,\eta\,\in \,\t_{e,R}^*}\,\g_{R,\,\eta},\qquad\ \g_{R,\,\eta}\,=\,\{x\in\g_R\,|\,\,[t,x]=\eta(t)x\, \text{ for all }\  t\in\t_{e,R}\}.\end{equation}
Moreover, each weight space $\g_{R,\,\eta}$ is a free $R$-module of finite rank and
$$\g_{R,0}\,=\,\l_{R}\,=\,\t_{e,R}\oplus [\l_R,\l_R].$$
The Lie ring $[\l_R,\l_R]$ is freely generated as an $R$-module by all $e_\beta$ with $\beta\in\Phi_I$ and all $h_{\alpha_i}$ with $i\in I$. 
We denote by $\Phi_e$ the set of all nonzero $\eta\in\t_{e,R}^*$ such that
$\g_{R,\,\eta}\ne \{0\}$. It is straightforward to see that each $\eta\in\Phi_e$ is obtained by restricting
a root $\beta\in\Phi$ to $\t_{e,R}$. 
Since all bad primes of $\Phi$ are invertible in $R$ it follows from basic properties of root systems that that 
$\eta(t_i)\in R^\times\cup\{0\}$ for all $\eta\in\Phi_e$ and all $i\in  \{1,\ldots,l\}\setminus I$. 
The adjoint action of $\lambda_e:=\lambda_{I,J}\in X_*(T_\Z)$ on $\g_R$ induces $\Z$-gradings on all direct summands $\g_{R,\,\eta}$. 

Since  $\g_R^e:=\g_R\cap\g_\Q^e$
and all $\g_{R,\,\eta}$ are $(\ad\,e)$-stable and $\lambda_e$ is optimal for $e\in\g_\C$, we have that
$$\g_R^e\,=\, \textstyle{\bigoplus}_{i\in\Z_+,\,\eta\in\Phi_e}\,\g_{R,\,\eta}^e(i).$$
Since $\ad\,e$ is nilpotent, $\g_{R,\,\eta}^e\ne \{0\}$ for every $\eta\in\Phi_e$.

Let $(\,\cdot\,|\,\cdot\,)$ be the $W$-invariant scalar product on the Euclidean space $\mathbb R\Phi$ such that $(\alpha|\alpha)=2$ for every short root $\alpha\in\Phi$. Put $d:=(\tilde{\alpha}|\tilde{\alpha})/(\alpha_0|\alpha_0)$ where $\tilde{\alpha}$ and $\alpha_0$ are the highest root and  the maximal short root in $\Phi_+(\Pi)$, respectively.
We take for $\kappa$ the {\it normalised Killing form} on $\g_\C$ which has the property that
$$\kappa(h_\alpha,h_\beta)\,=\,\frac{4d(\alpha|\beta)}
{(\alpha|\alpha)(\beta|\beta)}\  \  \ \text{ and } \ \ \kappa(e_\alpha,e_{-\beta})\,=\,\delta_{\alpha,\beta}\cdot \frac{2d}{(\alpha|\alpha)}\qquad\ (\forall\,\alpha,\beta\in\Phi).$$
This form is $\Z$-valued on $\g_\Z$ and
induces an $(\Ad\, G_\k)$-invariant 
symmetric bilinear form  on $\g_\k$ for any
algebraically closed field $\k$ such that the radical of $\g_\k$ coincides with the centre; see \cite[Lemma~2.2]{Pr97}. Of course, it is well-known that if $\z(\g_\k)\ne\{0\}$ then either $p$ is a bad prime for $\Phi$ or $G$ is of type ${\rm A}_{kp-1}$ for some $k\ge 1$.

In our situation, all bad primes of $\Phi$ are invertible in $R$ and $\Phi$ is not of type $\rm A$. Therefore, it is straightforward to see that the determinant of
the Gram matrix of $\kappa$ with respect to a Chevalley basis of $\g_\Z$ is invertible in $R$. We can use $\kappa$ to identify $\g_R$ with the dual $R$-module $\g_R^*$. This means that for every $\chi\in\g_R^*$ there exists a unique $x\in \g_R$ such that $\chi=\kappa(x,-)$.
Given an $R$-submodule $E$ of $\g_R$ we set $E^\perp:=\{x\in\g_R\,|\,\,\kappa(x,E)=0\}$. 

Out next lemma is well known, but since it is used several times in what follows a short proof is included for the reader's convenience.
\begin{Lemma}\label{R-mod} Let $V$ be a free $R$-module of rank $n$
and	let $W$ be a nonzero $R$-submodule of $V$ generated by elements $w_1,\ldots, w_m\in V$.
Suppose further that for every prime $p$ of $R$ the subspace of the $\mathbb{F}_p$-vector space $V\otimes_R\mathbb{F}_p\cong V/pV$ spanned by the cosets $w_1+pV,\ldots, w_m+pV$ has dimension equal to $\dim_\Q W_\Q$ where $W_\Q$ is the $\Q$-span of $W$ in $V\otimes_R \Q$. Then $W$ is a direct summand of $V$.
\end{Lemma}
\begin{proof}	
In order to prove the lemma we are going to use the Smith normal form of rectangular matrices over $R$.
Expressing the generators $w_i$ of $W$ via a free basis of $V$ one obtains a rectangular matrix $M\in\Mat_{m\times n}(R)$. As $R$ is a principal ideal domain there exist invertible $X\in \Mat_m(R)$ and $Y\in \Mat_n(R)$ such that the rectangular matrix $XMY$ is diagonal. It follows that
there exist a free basis $b_1,\ldots, b_n$ of $V$ and elements $r_1,\ldots, r_n\in R$ such that the $R$-module
$W$ is generated by $r_1b_1,\ldots, r_nb_n$. We may assume further that there exists $d\in\{1,\ldots, n\}$ such that $r_i\ne 0$ for
$i\le d$ and $r_i=0$ for all $i>d$. Then $d=\dim W_\Q$.
Since for every prime $p$ of $R$ 
the image of $W$ in $V/pV$ has dimension $d$ over $\mathbb{F}_p$, it must be that $b_i\in R^\times$ for all $1\le i\le d$. As a consequence, $W$ is a direct summand of $V$.
\end{proof}
Our next lemma is a refined version of Spaltenstein's result proved in \cite{Sp}.
\begin{Lemma}\label{c_g-R-e} The following are true:
	\begin{enumerate}
		\item [(i)\,] $\g_{R,\,\eta}^e$ is a direct summand of $\g_{R,\,\eta}$ for every $\eta\in\Phi_e$;
		
		\smallskip
		
		\item[(ii)\,] $[e,\g_R(i)]\,=\, \g_R(i+2)$ for all $i\ge 0$;
		
		\smallskip
		
		\item[(iii)\,] $[e,\g_R]\,=\,(\g_R^e)^\perp$ is a direct summand of $\g_R$.
	\end{enumerate}
	\end{Lemma}
\begin{proof}
If $r\in R\setminus\{0\}$ and $x\in\g_{R,\,\eta}$ are such that
$rx\in \g_{R,\,\eta}^e$ then $r[e,x]=[e,rx]=0$ forcing $x\in\g_{R,\,\eta}^e$.
This shows that the finitely generated $R$-module $\g_{R,\,\eta}/\g_{R,\,\eta}^e$
is torsion free. As $R$ is a principal ideal domain this implies that $\g_{R,\,\eta}/\g_{R,\,\eta}^e$ is a free $R$-module of finite rank.
So there exist $v_1,\ldots, v_s\in \g_{R,\,\eta}$ whose cosets modulo $\g_{R,\,\eta}^e$ form a free $R$-basis of $\g_{R,\,\eta}/\g_{R,\,\eta}^e$.
Let $V_\eta$ be the $R$-submodule of $\g_{R,\,\eta}$ generated by 
the $v_i$'s. Then $\g_{R,\,\eta}\,=\,V_\eta\oplus \g_{R,\,\eta}^e$ which proves part~(i).

In order to prove (ii) we apply Lemma~\ref{R-mod} to the free $R$-module
$V=\g_R(i+2)$ and its submodule $W=[e,\g_R(i)]$.
Let $p$ be a prime of $R$ and let $\k$ be an algebraically closed field of characteristic $p$. Then $p$ is a good prime for $\Phi$ and it follows from \cite[Theorem~A(ii)]{Pr03} that
$$\dim_\k[e_\k,\g_\k]\,=\,\dim_\C[e,\g_\C]\,=\,\dim_\C\g(i+2)\,=\,\dim_\k\g_\k(i+2).$$
Since this holds for any prime of $R$, Lemma~\ref{R-mod} entails that $[e,\g_R(i)]\,=\,\g_R(i+2)$.

It follows from part~(i) that $\g_R^e$ is a direct summand of $\g_R$. In view of our earlier remarks this implies that so is $(\g_R^e)^\perp$. The invariance of $\kappa$ yields 
$[e,\g_R]\subseteq (\g_R^e)^\perp$. As $[e,\g_\C]\,=\, (\g_\C^e)^\perp$  and $[e_\k,\g_\k]\,=\,(\g_\k^{e_\k})^\perp$ have the same dimension (for any algebraically closed field $\k$ whose characteristic is good for $\Phi$) we can again invoke the Smith normal form to conclude that $[e,\g_R]\,=\, (\g_R^e)^\perp$. This completes the proof.
\end{proof}
\begin{Remark}
It follows from Lemma~\ref{c_g-R-e} 
that each $\g_{R}^e(i)=\g_R^e\cap\g_R(i)$ is a direct summand of $\g_R(i)$. In particular, the Lie ring $\g_{R}^e(0)$ is a direct summand of $\g_R(0)$ and an $R$-form of the reductive Lie algebra
$\g_\Q^e(0)\,=\,\g_\Q^e\cap \g_\Q(0)$.
By Proposition~\ref{e-Delta}(iv), the latter Lie algebra is $\Q$-split.
Furthermore, $\t_{e,R}=\g_{R,\,0}^e$ is an $R$-form of the maximal $\Q$-split toral subalgebra $\t_{e,\Q}$ of $\g_Q^e(0)$ and for every $\eta\in \Phi_e$ such that $\g_{R,\,\eta}^e(0)\ne \{0\}$
the direct summand $\g_{R,\,\eta}^e(0)$ of  $\g_R^e(0)$ is a free $R$-module of rank $1$. Combining the above with Lemma~\ref{c_g-R-e}(iii) one observes that
the $R$-form $\g_R^e(0)$ of the reductive Lie algebra $\g_\Q^e(0)$ is split over $R$ and $\g_R(0)\,=\,\g_R^e(0)\oplus [e,\g_R(-2)]$.

\end{Remark}	
\section{Generalised Gelfand--Graev modules over rings and their endomorphisms}
\label{ss:W-R}
\subsection{Forms of generalised Gelfand--Graev modules}
\label{ss:RformsQ}  Recall the ring $R$ from Subsection~\ref{ss:standardhypotheses}. Our deliberation in the previous Subsection will enable
us to construct $R$-forms, $Q_R$, of the generalised Gelfand--Graev modules $Q$ defined in Subsection~\ref{ss:finiteWalgebras}. 
We fix $(I,J)\in\P(\Pi)$, choose $e:=e_{I,J}$ as in Subsection~\ref{ss:char0nilp}
and denote by $e_\k$ the image of $e$ in $\g_\k=\g_R\otimes_R\k$.
To ease notation we write $L_A$ for $L_{J,A}$ and set $\l_A:=\Lie(L_{J,A})$. Also let $\kappa$, $e$, $\lambda_e$, $t_i$ and $\Phi_e$  be as in Subsection~\ref{ss:geR}. 
We write $\Phi_e^+$ (resp. $\Phi_e^-$) for 
the set of all $\eta\in \Phi_e$ such that $\eta(t_i)>0$ (resp. $\eta(t_i)<0$)
for some $i\not\in I$. Our remarks in Subsection~\ref{ss:geR} imply that
$\g_R$  decomposes as $$\g_R\,=\,\n_{-,R}\oplus \l_R\oplus \n_{+,R}$$ where $\n_{\pm,R}$ is the $R$ span of all
$\g_{R,\,\eta}$ with $\eta\in\Phi_e^\pm$. This is a triangular decomposition of the Lie ring $\g_R$ associated with its Levi subring $\l_R$ and the normalised Killing form $\kappa$ vanishes on $\n_{\pm,R}$. The optimal cocharacter $\lambda_e=\lambda_{I,J}$ lies in $X_*(L_R)$ and hence induces $\Z$-gradings on all components of this decomposition. Since
$e$ is distinguished in $\l_R$ we have that $\l_R(i)=\{0\}$ for all odd $i$, forcing $$\g_R(-1)\,=\,\n_{-,R}(-1)\oplus \n_{+,R}(-1).$$ 	

As before, we set $\chi:=\kappa(e,-)$ and regard $\chi$ as an element of the dual $R$-module $\g_R^*$. 
Keeping the notation of Subsection~\ref{ss:PBWtheorem} we define
$\Psi\colon\,\g_R(-1)\times \g_R(-1)\to R$ by setting $\Psi(x,y)=\chi([x,y])$
for all $x,y\in \g_R(-1)$. This is a skew-symmetric $R$-bilinear form on $\g_R(-1)$ with values in $R$. Since $e\in\l_R$ and $\n_{\pm,R}$ are orthogonal to $\l_R$ with respect to $\kappa$ the direct summands $\n_{\pm,R}(-1)$ of $\g_R(-1)$ are totally isotropic with respect to $\Psi$.
Let $B_-=\{z_1',\ldots,z_s'\}$ and $B_+=\{z_1,\ldots,z_s\}$ be $R$-bases of $\n_{-,R}(-1)$ and  $\n_{+,R}(-1)$, respectively. The Gram matrix of $\Psi$ with respect to the ordered $R$-basis $B_-\cup B_+$ of $\g_R(-1)$ has form
$$
\begin{bmatrix} 
O_s& M \\
-M^\top & O_s\\
\end{bmatrix}
$$ for some square matrix $M=(M_{ij})_{1\le i,j\le s}$ with entries in $R$, where $O_s$ denotes the $s\times s$ zero matrix.

Let $p$ be an arbitrary good prime for $\Phi$ and $\k$ an algebraically closed field of characteristic $p>0$. Let $M_\k\in\Mat_s(\k)$ be the matrix such that for all $1\le i,j \le s$ the $(i,j)$-th entry equals the image of $M_{ij}$ in $R/pR\subset\k$. 
Since the cocharacter $\lambda_e$ is  optimal for $e_\k$ for any field $\k$ as above we have that $\g^e_{\k}\cap\g_\k(-2)\,=\,\{0\}$. This implies
that the skew-symmetric bilinear form $\Psi_\k$ on $\g_\k(-1)$ obtained from $\Psi$ by reduction modulo $p$ is non-degenerate. As a consequence, the matrix $M_\k$ is invertible in arbitrary good characteristic. This forces 
$\det(M)\in R^\times$ showing that 
$M^{-1}\in\Mat_s(R)$. Since
$$\begin{bmatrix} 
M^{-1}& O_s \\
O_s & I_s\\
\end{bmatrix}
\begin{bmatrix} 
O_s& M \\
-M^\top& O_s\end{bmatrix}	
	\begin{bmatrix} 
	(M^{-1})^\top& O_s \\
	O_s& I_s\end{bmatrix}
	\,=\,\begin{bmatrix} 
	O_s& I_s \\
	-I_s & O_s\\
	\end{bmatrix},$$
we may adjust $B_-\subset \n_{-,R}$ in such a way that
that $\Psi(z_i',z_j)=\delta_{i,j}$. As this adjustment does not affect
our choice of $B_+$ we may also assume that $z_i=e_{\gamma_i}$ for some positive roots $\gamma_i\in \Phi$, where 
$1\le i \le s$.

We now set $\g_R(-1)_0\,:=\,\n_{-,R}(-1)$ and let $\m_R$ denote the Lie subring $\g_R(-1)_0\oplus\bigoplus_{i<1}\g_R(i)$ of $\n_{-,R}$. We write 
$\m_{R,\chi}$ for the free $R$-submodule of $U(\g_R)$ generated
by all $x-\chi(x)$ with $x\in \m_R$ and define 
$$Q_R\,:=\,U(\g_R)/U(\g_R)\m_{R,\chi}.$$
Since the PBW theorem holds for $U(\g_R)$ it is straightforward to see that
$U(\g_R)\m_{R,\chi}$ is a direct summand of $U(\g_R)$ and
$Q$ is a free $R$-module. In view of Lemma~\ref{c_g-R-e} there exists a basis $\{x_1,\ldots, x_m\}$ of the free $R$-module $\bigoplus_{i\ge 0}\g_R(i)$ such that $x_1,\ldots, x_r$ is a free basis of $\g_R^e$ and $x_i\in \g_{R,\,\eta_i}(n_i)$ for some $n_i\in\Z_+$ and $\eta_i\in\Phi_e\cup\{0\}$.
Therefore, we may (and will) assume that $Q_R$ is freely generated by that PBW monomials
$x^\i z^\j = x_1^{i_1}\cdots x_m^{i_m} z_1^{j_1}\cdots z_s^{j_s}$ 
with $\i\in\Z_+^m$ and $\j\in\Z_+^s$.

 Given an algebraically closed field $\k$ whose characteristic is good for $\Phi$ we put 
$Q_\k\,:=\,Q_R\otimes_R \k$. Note that $Q_\k$ is nothing but the module $Q$ introduced in Subsection~\ref{ss:finiteWalgebras}. 
All of the above is consistent with the conventions of \cite[\textsection 3]{Sasha1} which we are going to use in what follows. 

We define
the Kazhdan filtration of $Q_R$ as in Subsection~\ref{ss:PBWtheorem}.
Since the Lie ring $\m_R$ and the left ideal $U(\g_R)\m_{R,\chi}$ are stable under the adjoint action of $\t_R$, we have natural actions of $\t_R$ on the $R$-module $Q_R$ and the $R$-algebra $Q_R^{\ad\, \m_R}$. Our choice of $x_i$ with $1\le i\le m$ and $z_j$ with $1\le j\le s$ ensures that all PBW monomials $x^\i z^\j\in Q_R$ are weight vectors for $\t_R$. In particular, all components $\F_d(Q_R)$ of the Kazhdan filtration of $Q_R$ are $\t_R$-stable.

\subsection{Arithmetic properties of generators of finite $W$-algebras}
Let $k\in\{1,\ldots, r\}$ and recall that $x_k\in\g_{R,\,\eta_k}^e$ gives rise to a PBW generator $\Theta(x_k)$ of $U(\g_\C,e)$ given by the formula (\ref{e:PBWKazhdanterm}).
By \cite[4.6]{Sasha1}, all coefficients $\lambda_{\i,\j}^k$ of $\Theta(x_k)$ are rational numbers. Since $R$ is a principal ideal domain there is a positive integer $D_k$ such that
all $D_k\lambda_{\i,\j}^k$ are in $R$ and the ideal of $R$ generated by these elements (with $k$ fixed) equals $R$. 
Put
$\widehat{\Theta}(x_i):=D_k\Theta(x_k)$. By our choice of $x_i$'s and $z_j$'s in Subsection~\ref{ss:RformsQ}, we have that $\widehat{\Theta}(x_k)\in Q_R^{\ad\,\m_R}$.

Recall that $\Phi$ is an irreducible root system of type other than $\rm A$.
\begin{Proposition}\label{ThetaR}
Let $\ell\in \N$ be the smallest good prime for $\Phi$ and suppose that
$x_k\in \g_{R}^e(n_k)$ where $n_k\le \ell-2$. Then $D_k=1$ and $\Theta(x_k)\in Q_R^{\ad\,\m_R}$.
\end{Proposition}
\begin{proof} 
As $D_k$ is a positive integer it suffices to show that $D_k\in R^\times$. So suppose the contrary and let $p\in \N$ be a prime of $R$ such that $p\mid D_k$. Then $p$ is a good prime for $\Phi$. 
Let $\k$ be an algebraically closed field of characteristic $p$ and write $\bar{\chi}$ for the linear function on $\g_\k$ induced by $\chi\in\g_R^*$. Let $I_{\bar{\chi}}$ be the two-sided ideal of $U(\g_\k)$ generated by all 
$x^p-x^{[p]}-\bar{\chi}^p(x)$ with $x\in\g_\k$ and define
$$Q_{\k,\,\bar{\chi}}\,:=\,U(\g_{\k})/\big(U(\g_\k)\m_{\k,\bar{\chi}}+
I_{\bar{\chi}})$$ where $\m_{\k,\bar{\chi}}$ is the subspace of $U(\g_\k)$ spanned by all $x-\bar{\chi}(x)$ with $x\in \g_\k$. 
As explained in \cite[\textsection 2]{Sasha1}, the $\g_\k$-module 
$Q_{\k,\,\bar{\chi}}$ is a projective generator for the reduced enveloping algebra $U_{\bar{\chi}}(\g_\k)$. Furthermore, the modular finite $W$-algebra $U_{\bar{\chi}}(\g_\k,e_\k)$ from Subsection~\ref{ss:pcentreandSkryabin} identifies with the endomorphism algebra 
$\End_{\g_\k}(Q_{\k,\,\bar{\chi}})^{\rm op}$ (see also \cite[Proposition~8.7]{GT18}).

Let $\{x_i\,|\,\,1\le i\le m\}\subset \bigoplus_{i\ge 0}\g_R(i)$ and $\{z_j\,|\,\,1\le j \le  s\}\subset \g_R(-1)$ be as in Subsection~\ref{ss:RformsQ}. Given $x\in \g_R$ we write $\bar{x}$ for the image of $x$ in 
$\g_\k=\g_R\otimes_R\k$. 
For $\i\in\Z_+^m$ and $\j\in\Z_+^s$ put $\bar{x}^\i\bar{z}^\j\,:=\, \bar{x}_1^{i_1}\cdots\bar{x}_m^{i_m}\bar{z}_1^{j_1}\cdots\bar{z}_s^{j_s}$ and regard each $\bar x^\i \bar z^\j$ as an element of $Q_{\k,\,\bar{\chi}}$.
It follows from \cite[\textsection 3]{Sasha1} that the set 
\begin{equation}
\label{bars}
\{\bar{x}^\i\bar{z}^\j
\,|\,\, 0\le i_k,j_k\le p-1\} 
\end{equation}
is a $\k$-basis of $Q_{\k,\,\bar{\chi}}$. It is immediate from our construction of $Q_R$ that there exists a natural $\g_\k$-module homomorphism $\mu\colon\, Q_R \otimes_R \k\to  Q_{\k,\,\bar{\chi}}$
such that $\mu(x^\i z^\j\otimes_R 1)=\bar{x}^\i\bar{z}^\j$ for all $\i\in \Z_+^m$ and $\j\in \Z_+^s$ with $0\le i_k,j_k\le p-1$ . Evidently $\mu$ is surjective and its kernel satisfies:
\begin{enumerate}
\item[(i)\,] $\F_d(Q_R\otimes_R \k)\cap \ker \mu\, = \,\{0\}$ for $d < p$;

\smallskip

\item[(ii)\,] $\F_{p}(Q_R\otimes_R \k)\cap \ker \mu\,  = \,\sum_{i=1}^s \k z_i^p$.
\end{enumerate}

Note that $\widehat{\theta}_k\,:=\,\mu(\widehat{\Theta}(x_k))$ is a well-defined element of $U_{\bar{\chi}}(\g_\k,e_\k)$.
Since $D_k\in pR$ it follows from (\ref{e:PBWKazhdanterm}) that
the expression of $\widehat{\theta}_k$ in terms of the PBW basis
(\ref{bars}) has no terms of the form $a_{\i,\j}\bar{x}^\i\bar{z}^\j$ with $i_{r+1}=\cdots=i_m=0$, $\j=\bf{0}$, and $a_{\i, \j}\in\k$. Applying \cite[Lemma~3.2]{Sasha1} we now deduce that $\widehat{\theta}_k=0$. 

If $n_k<\ell-2$ or $n_k=\ell-2$ and $p>\ell$ then (i) implies that $\widehat\Theta(x_k)\otimes_R 1 = 0$ in $Q_R \otimes_R \k$. However, our choice of $D_k$ ensures that $\widehat{\Theta}(x_k)\otimes_R 1$ is a nonzero element of $Q_R\otimes_R\k$, and the contradiction proves the Proposition in this case.

It remains to consider the case where $n_k=\ell-2$ and $p=\ell$. Then $|\i|+|\j|\le |(\i,\j)|_e=p$ holds for every PBW monomial $x^\i z^\j$ involved in $\widehat{\Theta}(x_k)$. Since $\widehat{\theta}_k=0$, property (ii) of $\ker \mu$ entails that  $\widehat{\Theta}(x_k)\otimes_R 1$ is a linear combination of elements $z_i^p\in Q_\k$ where $1\le i\le s$. Recall that $x_k\in\g_{R,\,\eta_k}^e$ for some $\eta_k\in\Phi_e\cup\{0\}$ and that $z_1,...,z_s$ are root vectors for $\t_\k$. As a consequence,
$\widehat{\Theta}(x_k)\otimes_R 1$ has weight $0$ with respect to the adjoint action of $\t_{e,\k}$ on $Q_\k$. If $\eta_k=\beta_k\vert_{\t_{R,e}}$ for some $\beta_k\in \Phi$ then  $\beta_k(t_j)\in R^\times$ for some $j\not\in I$ by our discussion in Subsection~\ref{ss:geR}. But in this case $t_j\otimes_R 1$ cannot annihilate  $\widehat{\Theta}(x_k)\otimes_R 1$. So it must be that $\eta_k=0$.
But then $x_k\in \l_R(n_k)$ where $\l_R$ the centraliser in $\g_R$ of $\t_{e,R}$. Since $e$ is distinguished in $\l_\C$ we have that
$\l_R(i)=\{0\}$ for all odd $i$. Since $n_k=\ell-2$ is odd (as $\Phi$ is not of type $\rm A$) we reach a contradiction thereby proving the current Proposition.
\end{proof}
\begin{Remark}
If $\g=\gl_N$ then every
nilpotent element $e\in\g$ is Richardson and admits an {\it even} good grading
(it has the property that $\g(i)=\{0\}$ for all odd $i$).
In \cite{BK}, Brundan and Kleshchev used such gradings to construct integral forms of all finite $W$-algebras $U(\gl_N,e)$ (see also \cite[\textsection 3.4]{GT19b}).
It is worth remarking that outside type $\rm A$ the nilpotent elements contained in a single sheet of $\g$ and induced from nonzero rigid nilpotent elements of Levi subalgebras of $\g$ do not admit even good gradings. This follows from the classification of sheets in the Lie algebras
of standard reductive groups; see \cite[Theorem~2.8]{PS18}.
\end{Remark}
Let $x\in \g_{\C}^e(n)$. When $n\ge 2$, explicit expressions for $\Theta(x)$ are too bulky to write down, but there exist reasonably short
formulae in the case where $x\in \g_\C^e(0)\cup \g_\C^e(1)$.  They  appeared in \cite{Pr07a}, in a slightly different setting. For $n = 0$ we let
\begin{eqnarray}
\label{e:zerogenerators}
\Theta(x)\, :=\, x - \frac{1}{2}\sum_{i=1}^{s}\, z_i[x, z_i'] \in Q_\C
\end{eqnarray}
and for $n = 1$ we let
\begin{eqnarray}
\label{e:onegenerators}
\Theta(x) &:=& x - \sum_{i=1}^s [x, z_i'] z_j + \frac{1}{3} \sum_{i,j = 1}^s [[x, z_i'], z_j'] z_j z_i\\ 
\nonumber& & - \frac{1}{3} \sum_{i, j=1}^s \big(\kappa(e,[z_j', [x, [z_j, z_i']]]) - \kappa(e,[z_j, [x, [z_j', z_i']]]\big)z_i \in Q_\C
\end{eqnarray}
\begin{Lemma} \label{L:JEMS} The elements in (\ref{e:zerogenerators}) and (\ref{e:onegenerators}) lie in $U(\g_\C,e)$. If $u\in \g^e(0)$ and  $v\in \g_\C^e(0)\cup \g_\C^e(1)$ then 
$[\Theta(u),\Theta(v)]\,:=\,\Theta(u)\Theta(v)-\Theta(v)\Theta(u)\,=\,\Theta([u,v])$.
\end{Lemma}
\begin{proof}
Set $\n_1\, := \,\bigoplus_{i<-1} \g_\C(i)$ and  $\n_0\, :=\, \bigoplus_{i<0} \g_\C(i)$ and denote by $\n_{1,\chi}$ the subspace of $U(\g_\C)$ spanned by all $x - \chi(x)$ with $x\in \n_1$. 
Let $\widetilde{Q}_\C\, := \,U(\g_\C) / U(\g_\C)\n_{1,\chi}$.
In \cite[\textsection 5.5]{GG02} it was demonstrated that the canonical homomorphism $\pi\colon\,\widetilde{Q}_\C \twoheadrightarrow Q_\C$ induces an algebra isomorphism $\widetilde Q_\C^{\ad\,\n_0} \stackrel{\sim}{\rightarrow} U(\g_\C, e)$. In \cite[2.5]{Pr07a} certain formulae for elements of $\widetilde Q_\C^{\ad\,\n_0}$ were presented, associated to the elements 
$u,v\in\g_\C(0)\cup \g_\C^e(1)$. It is straightforward to check that those elements are sent to \eqref{e:zerogenerators} and \eqref{e:onegenerators} respectively by the map $\pi$. The final claim follows from \cite[Lemmas~2.4 \& 2.5]{Pr07a}.
\end{proof}
\begin{Remark} If $x\in \g_R^e(0)$ then $\Theta(x)\in Q_R$ because $\frac{1}{2}\in R$ (recall that $\Phi$ is not of type $\rm A$). The same holds for $x\in \g_R^e(1)$ when $\frac{1}{3}\in R$. Note that $\frac{1}{3}\not\in R$ if and only if $\Phi$ is of type $\rm B$, $\rm C$ or $\rm D$
and in these cases it is not immediately clear from (\ref{e:onegenerators}) that $\Theta(x)\in Q_R$. However, a closer look at the formula reveals that 
this does hold for all weight vectors $x\in \g_{R,\,\eta}^e(1)$ such that $\eta=\beta\vert_{\t_R}$ and $\beta\in\Phi_-(\Pi)$. One can use an explicit combinatorial description of $\g_R^e$ to show that  $\g_R^e(1)$ is generated by such weight vectors as a module over $\g_R^e(0)$. Applying Lemma~\ref{L:JEMS} to the commutator action
of $\Theta(\g_R^e(0))\subset Q_R^{\ad\, \m_R}$ on 
$\Theta(\g^e_R(1))$ then yields $\Theta(\g^e_R(1))\subset Q_R$.
We omit the details because this inclusion follows from 
Proposition~\ref{ThetaR}. 
\end{Remark}

\subsection{Structural features of rigid centralisers $\g_R^e$} \label{ss:rigidR} 
Unless otherwise stated, we assume from now on that  $e=e_{I,J}$ is a nonzero rigid nilpotent element of $\g_\C$ and we retain the conventions introduced earlier. Our next goal is to show that the $R$-algebra 
$Q_R^{\ad\, \m_R}$ admits a nice PBW basis. This will imply that the unital $R$-algebra $U(\g_R,e)\,:=\,Q_R^{\ad\, \m_R}$ satisfies all requirements of 
the first part of Theorem~\ref{main1}. 
Our arguments will rely on
Proposition~\ref{ThetaR}, Lemma~\ref{L:JEMS} and some structural properties of the Lie ring $\g_R^e$.

If $\t_{e,\C}$ is a maximal torus of $\g_\C^e$ then Bala--Carter theory tells us that the weights of the Dynkin cocharacter on $\c_{\g_\C}(\t_{e,\C})$ are all even. It follows that $\t_{e,\C}$ acts with nonzero weights on $\g_\C^e(1)$, so that $\g_\C^e(1)\,=\,[\g_\C^e(0),\g^e_\C(1)]$, a fortiori.
If $\Phi$ is of type $\rm B$, $\rm C$ or $\rm D$ then it is proved in
\cite{Ya10} that the Lie algebra $\g_\C^e(0)$ is semisimple and the nilradical
$\bigoplus_{i>0}\,\g_\C^e(i)$ of $\g_\C^e$ is generated by 
$\g_\C^e(1)$. Combining with $\g_\C^e(1)\,=\,[\g_\C^e(0),\g^e_\C(1)]$, this implies that $\g_\C^e$ is perfect.
Using computational methods de Graaf checked in \cite{dG} that this continues to holds for the majority of rigid nilpotent elements in exceptional Lie algebras $\g_\C$. More precisely, the Lie algebra
$\g_\C^e(0)$ is semisimple in all rigid cases and  
$\g_\C^e(1)\,=\,[\g_\C^e(0),\g^e_\C(1)]$ generates the Lie algebra $\bigoplus_{i\ge 0}\,\g_\C^e(i)$ unless the Bala--Carter label of $e$ is listed in Table~\ref{t:orbittable}, in which case $\g_\C^e\,=\,\C e\oplus  
[\g_\C^e,\g_\C^e]$ implying that $[\g_\C^e(1),\g_\C^e(1)]$ has codimension $1$ in $\g_\C^e(2)$. We remark that $e\in [\g_\C^e(1), \g_\C^e(1)]$ if $e \in [\g_\C^e, \g_\C^e]$ by a simple application of Weyl's complete reducibility theorem, viewing $\g_\C^e(2)$ as a $\g_\C^e(0)$-module.



\begin{table}[htb]
\bgroup
\def\arraystretch{1.5}
 \begin{tabular}{| c | | c | c | c | c | c | c |} 
 \hline
 \textnormal{Type of $\Phi$} &  ${\sf G_2}$ &  ${\sf F_4}$ & ${\sf E_7}$ & ${\sf E_8}$ & ${\sf E_8}$ & ${\sf E_8}$ \\
 \hline 
\textnormal{Bala--Carter label of $e$} & $\widetilde{\sf A_1}$ & $\widetilde{\sf A_2} + {\sf A_1}$ & $({\sf A_3} + {\sf A_1})' $ & ${\sf A_3} + {\sf A_1}$ &${\sf A_5} + {\sf A_1}$  & ${\sf D_5}({\sf a_1}) + {\sf A_2}$\\
 \hline 
\end{tabular}
\egroup\vspace{6pt}
 \caption{Rigid orbits with imperfect centralisers.}
 \label{t:orbittable}
 \end{table}
If $\k$ is an algebraically closed field and $p={\rm char}(\k)$ is a good prime for $\Phi$ then all results mentioned above continue to hold for the rigid nilpotent elements $e_\k=e\otimes_R 1\in \g_\k$. Indeed, when $\Phi$ is of type 
$\rm B$, $\rm C$ or $\rm D$ 
the equalities of $\g_\k^{e_\k}(0)\,=\,[\g_\k^{e_\k}(0),\g_\k^{e_\k}(0)]$  and 
$\g_\k^{e_\k}(1)\,=\,[\g_\k^{e_\k}(0),\g_\k^{e_\k}(1)]$ follow
from \cite[Theorem~3(i)]{PT14} and the combinatorial description of rigid nilpotent orbits in orthogonal and symplectic Lie algebras
whilst the equality 
$\bigoplus_{i\ge 1}\,\g_\k^{e_\k}(i)\,=\,\langle \g_\k^{e_\k}(1)\rangle$
is proved in Proposition~\ref{P:arprop}. We mention in passing that
$\g_\k^{e_\k}(0)=\Lie(C(e_\k))$ where $C(e_\k)\,:=\,G_\k^{e_\k}\cap C_{G_\k}(\lambda_e)$ . By \cite[Theorem~A(iv)]{Pr03}, the group $C(e_\k)$
is reductive for every nilpotent element $e_\k\in\g_\k$. Therefore, the perfectness of 
of $\g_\k^{e_\k}(0)$ implies that in all rigid cases the group $C(e_\k)$ is semisimple.

If $\Phi$ is of type ${\rm G}_2$, ${\rm F}_4$, ${\rm E}_6$, ${\rm E}_7$ or ${\rm E}_8$ then \cite[Theorem~3.8]{PS18} shows that every rigid nilpotent element of $\g_\k$ is $(\Ad\,G_\k)$-conjugate to one of the elements $e_\k$ as above. A glance at the tables in {\it loc.\,cit.} reveals that 
the codimension of $[\g_\k^{e_\k}, \g_\k^{e_\k}]$ in $\g_\k^{e_\k}$ is 
independent of $p$ as long as $p$ is good for $\Phi$ and coincides with
the codimension of $[\g_\C^{e}, \g_\C^{e}]$ in $\g_\C^{e}$. This holds for all nilpotent elements $e_\k\in\g_\k$. 
If $e_\k$ is rigid then either $\g_\k^{e_\k}\,=\,[\g_\k^{e_\k}, \g_\k^{e_\k}]$ or the Bala--Carter label of $e_\k$ is listed in Table~\ref{t:orbittable} and $[\g_\k^{e_\k}, \g_\k^{e_\k}]$ had codimension $1$ in $\g_\k^{e_\k}$. According to \cite[4.4]{PS18} in the latter case $\g_\k^{e_\k}\,=\,\k e_\k\oplus [\g_\k^{e_\k}, \g_\k^{e_\k}]$ implying that $ [\g_\k^{e_\k}(1),\g_\k^{e_\k}(1)]$ has codimension $1$ in $\g_\k^{e_\k}(2)$. To justify the very last assertion we note that $\g_\k^e(0)$ acts completely reducibly on $\g_\k^e(2)$ for the six orbits listed in Table~1 (the latter is immediate from the description of $\g_\k^e(2)$ given in \cite[pp.~73, 76, 99, 129, 149, 150]{LT2}).

Most importantly for us, it was checked in \cite[4.4]{PS18} by computer-aided calculations that for all elements $e_\k$ listed in Table~\ref{t:orbittable} the Lie algebra $\bigoplus_{i\ge 1}\,\g_\k^{e_\k}(i)$ 
is generated by its graded components $\g_\k^{e_\k}(i)$ with $i\le \ell-2$ where $\ell$ is the smallest good prime of $\Phi$. Since this holds for almost all primes  $p\in \Z$ the same is true for $\g_\C^e$. Summarising we obtain the following.
\begin{Proposition}\label{P:R-derived} Suppose $\Phi$ is not of type $\rm A$ and let $\ell$ be the smallest good prime of $\Phi$. 
\begin{enumerate}
\item[(i)\,]		
If $e=e_{I,J}$ is an arbitrary nilpotent element and $i\ge 0$  then $[\g_R^e,\g_R^e]\cap \g_R(i)$ is a direct summand of the $R$-module $\g_R^e(i)$.

\smallskip

\item[(ii)\,] Suppose $e=e_{I,J}$ is rigid. If $e$ is not listed in Table~1 then $\g_R^e=[\g_R^e,\g_R^e]$ and
the Lie ring $\textstyle\bigoplus_{i>0}\,\g_R^e(i)$ is generated by $\g_R^e(1)$.

\smallskip
	
\item[(iii)\,]	If $e=e_{I,J}$ is listed in Table~1 then $\g_R^e=Re\,\oplus\, [\g_R^e,\g_R^e]$ and the Lie ring $\textstyle\bigoplus_{i>0}\,\g_R^e(i)$ is generated by the graded components $\g_R^e(i)$ with $i\le \ell-2$.
	\end{enumerate}
\end{Proposition}	
\begin{proof} By \cite[Theorem~3(i)]{PT14} and \cite[Tables~2 and 3]{PS18}, the equalities $$\dim_\k\big([\g_\k^{e_\k},\g_\k^{e_\k}]\cap \g_\k(i)\big)\,=\,\dim_\C\big([\g_\C^e,\g_\C^e]\cap\g_\C(i)\big),\quad\ i\geq 0,$$ hold for all algebraically closed fields $\k$ whose characteristic is a good prime for $\Phi$.	
Applying Lemma~\ref{R-mod} with $V=\g_R(i)$ and $W=[\g_R^e,\g_R^e]\cap\g_R(i)$ we obtain part~(i).

Suppose $e=e_{I,J}$ is rigid and its Bala--Carter label is not listed in Table~1.  
Since the Lie algebras $\g^e_\C(0)$ and $\g_\k^{e_\k}(0)$ are perfect by \cite{Ya10}, \cite{dG}, \cite{PT14} and \cite{PS18} we can apply part~(i) (with $i=0$ and $i=1$) to conclude that 
$\g_R^e(0)\,=\,[\g_R^e(0),\g_R^e(0)]$ and $\g_R^e(1)\,=\,[\g_R^e(0),\g_R^e(1)]$.
Easy induction on $n$ relying on  Proposition~\ref{P:arprop} (for classical types), 
\cite[Theorem~1.2]{PS18} (for the exceptional types) and part~(i)
implies that each $\g_R^e(n)$ with $n\ge 1$ is contained in the Lie subring of  $\textstyle\bigoplus_{i>0}\,\g_R^e(i)$ generated by $\g_R^e(1)$. 
This yields $\g_R^e\,=\,[\g_R^e,\g_R^e]$, proving (ii).

Finally, suppose that the Bala--Carter label of $e=e_{I,J}$ is listed in Table~1. Then it follows from \cite[4.5]{PS18} that the Lie algebra $\textstyle{\bigoplus}_{i>0}\,\g_\k^{e_\k}(i)$
 is generated by $\g_\k^{e_\k}(i)$ with $i\le \ell -2$. As this holds for infinitely many primes  the Lie algebra $\textstyle{\bigoplus}_{i>0}\,\g_\C^e(i)$ must enjoy the same property. Applying Lemma~\ref{R-mod} to suitably selected $R$-submodules of $\g_R(i)$ with $i>\ell-2$ we deduce that the Lie ring $\textstyle{\bigoplus}_{i>0}\,\g_R^e(i)$ is generated by the graded components $\g_R^e(i)$ with $i\le \ell-2$. 
 
 We know from \cite[p.~7]{dG} 
 and \cite[4.4]{PS18} that the Lie algebras $\g_\C^e(0)$ and $\g_\k^{e_\k}(0)$ are perfect, that 
 $\g_\C(2)\,=\,\C e\oplus [\g_\C^e(1),\g_\C^e(1)]$, and the equalities $\g_\k^{e_\k}(2)\,=\,\k e_\k\oplus [\g_\k^{e_\k}(1),\g_\k^{e_\k}(1)]$ hold for all algebraically closed fields $\k$ whose characteristic is a good prime for $\Phi$. Applying Lemma~\ref{R-mod} to suitable $R$-submodules of $\g_R^e(0)$ and $\g_R^e(2)$ we now deduce 
 that $\g_R^e(2)\,=\, R e\oplus [\g_R(1),\g_R(1)]$ and 
 $\g_R^e\,=\,R e\oplus [\g_R^e,\g_R^e]$.  This completes the proof. 
\end{proof}


\section{Rigid finite $W$-algebras over rings and Humphreys' conjecture}
In this section we are going to finish the proof of Theorem~\ref{main1} and Theorem~\ref{main}.
\subsection{Finite $W$-algebras $U(\g_R,e)$ associated with rigid nilpotent elements} 
We retain the notation introduced in Subsection~\ref{ss:PBWtheorem}. Our discussion in Subsection~\ref{ss:RformsQ} shows that all of the notation made there makes sense in the $\g_R$-module $Q_R\subset Q_\C$. In particular, we know that the $R$-module $Q_R$ has a free $R$-basis consisting of all monomials $x^{\i} z^{\j}$ with $\i\in \Z_+^m$ and $\j\in\Z_+^s$, where $x^\i=x_1^{i_1}\cdots x_m^{i_m}$ and $z^\j=z_1^{j_1}\cdots z_j^{j_s}$. Recall that the set $\{x_1,\ldots, x_m\}$ is  a homogeneous basis of the free $R$-module $\bigoplus_{i\ge 0}\,\g_R(i)$ and the first $r$ elements of this set form a free $R$-basis of $\g_R^e$, a direct summand of $\bigoplus_{i\ge 0}\,\g_R(i)$. If $i\in\{1,\ldots, r\}$ then $x_i\in \g_{R,\,\eta_i}^e(n_i)$ and the element $\Theta(x_i)$ from (\ref{e:PBWKazhdanterm}) lies in
$U(\g_\C,e)=Q_\C^{\ad\,\m_\C}$ and has Kazhdan degree $n_i+2$.

Any element $0\ne h\in U(\g_\C,e)$ of Kazhdan degree $n$ can be written as
$h=\textstyle{\sum}_{|(\i,\j)|_e\le n}\,\lambda_{\i,\j}(h)x^\i z^\j$ where $\lambda_{\i,\j}(h)\in \k\setminus\{0\}$ for some $(\i,\j)$ with $|(\i,\j)|_e=n$.
For $k\in\Z_+$ we put $$\Lambda_k(h):=\{(\i,\j)\in\Z_+^m\times \Z_+^s\,|\,\,\lambda_{\i,\j}(h)\ne 0
\,\text{ and }\, |(\i,\j)|_e=k\}$$
and let $\Lambda_k^0(h)$ be the set of all $(\i,\j)\in\Lambda_k(h)$ such that $\i\in\Z_+^r\times\{\mathbf{0}\}$ and $\j=\mathbf{0}$.
Denote by $\Lambda^{\max}(h)$ the set of all $(\mathbf{p},\mathbf{q})\in\Lambda_n(h)$ for which the total degree $|\mathbf{p}|+|\mathbf{q}|$ assumes its minimum value.
By \cite[Lemma~4.5]{Sasha1},  $\Lambda^{\max}(h)\subseteq\Lambda_n^0(h)$.
Given $\i=(i_1,\ldots,i_r)\in\Z_+^r$ we put $\Theta^\i:=\prod_{j=1}^r\Theta(x_j)^{i_j}$, an element of $U(\g_\C,e)$, and denote by $\varepsilon_k$ 
the multi-index 
$(\delta_{k,1},\ldots,\delta_{k,r})$.

For a rigid nilpotent element $e=e_{I,J}\in\g_R$ we define 
$$U(\g_R,e)\,:=\,Q_R^{\ad\, \m_R}\cong\,{\rm End}_{\g_R}(Q_R)^{\rm op}.$$

\begin{Proposition}\label{P:Theta-k} If $e=e_{I,J}$ is rigid then $\Theta(x_k)\in U(\g_R,e)$ for all $1\le k\le r$.
\end{Proposition}
\begin{proof} We use induction on $n_k\in \Z_+$.
If $n_k\le \ell-2$ the statement follows from Proposition~\ref{ThetaR}. Suppose now that $n_k>\ell-2$. There is a unique $\C$-linear injection  
$\Theta\colon\, \g_\C^e\to U(\g_\C,e)$ extending the assignment $x_i \mapsto \Theta(x_i)$. By Proposition~\ref{P:R-derived},
$x_k=\sum_{i=1}^t\,[u_i,v_i]$ for some $u_i\in\g_R^e(a_i)$ and
$v_i\in\g_R^e(b_i)$ where $a_i$ and $b_i$ are positive integers with $a_i+b_i=n_k$. 

Set $h_1:=\sum_{i=1}^t[\Theta(u_i),\Theta(v_i)]$. By our induction assumption, $h_1\in U(\g_R,e)$. Moreover, it follows from \cite[Theorem~4.6(iv)]{Sasha1} that 
$(\varepsilon_k,\mathbf{0})$ is the only element of total degree $1$ in $\Lambda^{\max}(h_1)$. If $({\i}(1),{\bf 0}),\ldots,({\i}(d),{\bf 0})$ are all elements of 
$\Lambda^{\max}(h_1)\setminus\{(\varepsilon_k,\mathbf{0})\}$
of the smallest total degree we set 
$h_2:=h_1-\sum_{j=1}^d\lambda_{\i(j),\bf{0}}(h_1)\Theta^{\i(j)}$. Since
$|\i(j)|\ge 2$ for all $j$ our induction assumption entails that
$h_2\in U(\g_R,e)$. The formula displayed in \cite[p.~27]{Sasha1}
shows that $(\varepsilon_k,\mathbf{0})\in\Lambda^{\max}(h_2)$ and
the smallest total degree of the elements in $\Lambda^{\max}(h_2)\setminus\{(\varepsilon_k,\mathbf{0})\}$ is bigger then that in $\Lambda^{\max}(h_1)\setminus\{(\varepsilon_k,\mathbf{0})\}$. 

Continuing the process started above we eventually arrive at an element $h_N\in U(\g_R,e)$ such that
$\Lambda^{\max}(h_N)=\{(\varepsilon_k,\mathbf{0})\}$. Since the set $\{(\i,\j)\in\Z_+^r\times \Z_+^s\,|\,\,|(\i,\j)|_e=n_k+2\}$ is finite this will require finitely many iterations.
If there is $q<n_k+2$ such that $\Lambda_q^0(h_N)\ne \emptyset$ and
$\Lambda_j^0(h_N)=\emptyset$ for all $q<j\le n-1$ then all terms of $h_N$
associated with $\Lambda_q^0(h_N)$ have Kazhdan degree $<n_k+2$. So we can 
use our induction assumption and safely clear all of them (including the terms of total degree $1$) by applying the same procedure. After finitely many steps we shall arrive at an element $h\in U(\g_R,e)$ with the property that
$\Lambda^{\max}(h)=\{(\varepsilon_k,\mathbf{0})\}$ and $\Lambda_j^0(h)=\emptyset$ for all $0\le j<n_k+2$. By the uniqueness in Lemma~\ref{L:PBWtheorem} we have $h=\Theta(x_k)$ and the proof is complete.
\end{proof}

\begin{Theorem}\label{T:part(i)}
Suppose $e=e_{I,J}$ is rigid and let $\k$ be an algebraically closed field whose characteristic is a good prime for $\Phi$.
\begin{enumerate}	
\item[(i)\,]
The set $\{\Theta^{\i}\,|\,\,\i\in\Z^r_+\}$ is a free basis of the $R$-module  $U(\g_R,e)$.

\smallskip

\item[(ii)\,] $U(\g_R,e)\otimes_R\C\,\cong\,U(\g_\C,e)$ as $\C$-algebras.

\smallskip

\item[(iii)\,] 
$U(\g_R,e)\otimes_R\k\,\cong\,U(\g_\k,e_\k)$ as $\k$-algebras.
\end{enumerate}
\end{Theorem}
\begin{proof} Let $U_R$
be the $R$-submodule of $U(\g_\C,e)$ spanned by all $\Theta^{\i}$ with $\i\in\Z_+^r$. It follows from Proposition~\ref{P:Theta-k} that $U_R\subseteq U(\g_R,e)$.
In order to prove that $U(\g_R,e)\subseteq U_R$ we are going to use induction on the Kazhdan degree of elements $h\in U(\g_R,e)$. Assume that  all elements $h\in U(\g_R,e)$ of Kazhdan degree $\le d$ lie in $U_R$ (this assumption is obviously true when $d=0$).

Now let $h\in U(\g_R,e)$ be an element of Kazhdan degree $d+1$ and let $n(h)$ be the total degree of the elements $(\i,\j)\in\Lambda^{\max}(h)$. Our discussion in Subsection~\ref{ss:RformsQ} shows that  $\F_d Q_R$, the $d$-th component of the Kazhdan filtration of $Q_R\subset Q_\C$, is a direct summand of $\F_{d+1}Q_R$. As $h\in Q_R$, combining \cite[Theorem~4.6(ii)]{Sasha1} with the formula displayed in \cite[p.~27]{Sasha1} implies that
$$h=\sum_{{\bf a}\in \Z_+^r,\,\,|{\bf a}|_e=d+1,\,\, |{\bf a}|=n(h)
}\,
\lambda_{\bf a}x^{\bf a}+
\sum_{|(\i,\j)|_e=d+1,\,|\i|+|\j|>n(h)}\,\lambda_{\i,\j}x^{\i}z^{\j}+h',$$
where all $\lambda_{\bf a}$ and $\lambda_{\i,\j}$ are in $R$ and $h'\in \F_{d}(Q_R)$. We now set $$h_1:=\,h-\sum_{{\bf a}\in \Z_+^r,\,|{\bf a}|_e=d+1,\, |{\bf a}|=n(h)
}\,
\lambda_{\bf a}\Theta^{\bf a}.$$
Clearly, $h_1\in \F_{d+1}U(\g_R,e)$. If $h_1\not\in \F_{d}U(\g_R,e)$ then the formula displayed in \cite[p.~27]{Sasha1} entails that the total degree of the elements $(\i,\j)\in\Lambda^{\max}(h_1)$ is bigger than $n(h)$.	
In this case we replace $h$ by $h_1$ and argue as before. After finitely many iterations we shall arrive at an element $h_N\in\F_d U(\g_R,e)$ such that
$h\equiv h_N\pmod {U_R}$. As $h_N\in U_R$ by our induction assumption we conclude that $U_R=U(\g_R,e)$. This proves part~(i).

Part~(ii) now follows immediately from (i) and \cite[Theorem~4.6(ii)]{Sasha1}. 
Thanks to Lemma~\ref{L:PBWtheorem}, in order to prove (iii) we just need to show that all elements $\Theta(x_i)\otimes_R 1\in Q_R\otimes_R\k$ with $1\le i\le r$ are invariant under the action of our unipotent group
$M_\k\subset G_\k=G_\Z(\k)$ on $Q_R\otimes_R \k$. Recall that $M_\k$ is generated by some root subgroups
$U_{\k,\beta}$ of $G_\k$. 
As explained in \cite[I.7.11]{Jan03}, the adjoint action of the $\Z$-group scheme $G_\Z$ gives rise to the natural action of the distribution algebra ${\rm Dist}(G_\Z)$ on $\g_\Z=\Lie(G_\Z)$ and $U(\g_\Z)$.
This implies that the unipotent group subscheme $M_\Z$ of $G_\Z$ generated by all $U_{\Z,\,\beta}$'s involved in the definition of $M$ acts on
$Q_R=U(\g_R)/U(\g_R)\m_{R,\chi}$.
Since each $\Theta(x_i)$ lies in $Q_R^{{\rm ad}\,\m_R}$ by Proposition~\ref{P:Theta-k}
all endomorphisms $\frac{1}{k!}(\ad\,e_\beta)^k$ with $k>0$ annihilate  
$\Theta(x_i)$. This implies that each $\Theta(x_i)\otimes_R 1$ lies in $U(\g_\k,e_\k)\,=\,(Q_R\otimes_R\k)^{\Ad\,M_\k}$ and we finish the proof by applying Lemma~\ref{L:PBWtheorem}. 
 \end{proof}
\subsection{Augmentation ideals of rigid finite $W$-algebras $U(\g_R,e)$}
In this Subsection we are going to prove Theorem~\ref{main1}(ii).
We continue to assume that $e=e_{I,J}$ is a rigid nilpotent element.
\begin{Proposition}\label{P:idealI1}
Suppose $e$ is not listed in Table~1 and let $I_R$ be the two-sided ideal of $U(\g_R,e)$ generated by all commutators $[u,v]=uv-vu$ with $u,v\in U(\g,e)$. Then $I_R$ is a free $R$-module and $U(\g_R,e)\,=\,R\,1\oplus I_R$.
\end{Proposition}
\begin{proof}
By \cite[Theorem~2]{PT14} and \cite[Theorem~A]{Pr14}, the algebra $U(\g_\C,e)$ admits a unique $1$-dimensional representation. This implies that the $\C$-saturation of $I_R$ is a two-sided ideal of codimension $1$ in $U(\g_\C,e)$. It follows that $I_R$ is a proper ideal of $U(\g_R,e)$. 
Thanks to Theorem~\ref{T:part(i)}(i), in order to prove the proposition it suffices to show that
for every $k\in\{1,\ldots, r\}$ there exists $c_k\in R$ such that $\Theta(x_k)-c_k\,1\in I$. This would ensure that the set
$$\big\{(\Theta(x_1)-c_1\,1)^{i_1}\cdots(\Theta(x_r)-c_r\,1)^{i_r}\,|\,\,(i_1,\ldots,i_r)\in\Z_+^r
\setminus\{(0,\ldots,0)\}\big\}$$
 is a free basis of the $R$-module $I_R$.
 
 We argue by induction on $n_k$.
 If $n_k\in\{0,1\}$ then  Proposition~\ref{P:R-derived}(ii) in conjunction with Lemma~\ref{L:JEMS} shows that $\Theta(x_k)\in I_R$. So  $c_k=0$ in this case. Suppose that $n_k>1$ and for every $i\le r$ with $n_i<n_k$ there is $c_i\in R$ such that $\Theta(x_i)-c_i\,1\in I_R$.
 By Proposition~\ref{P:R-derived}(ii) we have 
 $x_k=\sum_{i=1}^t\,[u_i,v_i]$ for some $u_i\in\g_R^e(a_i)$ and $v_i\in\g_R^e(b_i)$ where $a_i$ and $b_i$ are positive integers with $a_i+b_i=n_k$. At this point we repeat almost verbatim the argument used in the proof of Proposition~\ref{P:Theta-k} the only difference being that
 at each step of the reduction process outlined there we use $\Theta(x_i)-c_i\,1$ instead of 
 $\Theta(x_i)$. We stress that in the present case the reduction process will stop
 at an element of $\F_0U(\g_R,e)$, implying that
 $\Theta(x_k)-c_k\,1\in I_R$ for some $c_k\in R$.
 \end{proof}

It remains to consider the case where the Bala--Carter label of $e=e_{I,J}$ is listed in Table~1. By our choice of $\kappa$ the Casimir element
$$C\,=\, 2\sum_{\alpha\in\Phi_+(\Pi)}\,\frac{1}{\kappa(e_\alpha,e_{-\alpha})}e_\alpha e_{-\alpha}+\sum_{i=1}^l\,t_ih_{\alpha_i}+\sum_{\alpha\in\Phi_+(\Pi)}\,
\frac{h_\alpha}{\kappa(e_\alpha,e_{-\alpha})}$$ lies in the centre of $U(\g_R)$ and hence gives rise to a central element of $U(\g_R,e)\,=\,Q_R^{\ad\,\m_R}$. We identify $C$ with its image in $U(\g_R,e)$.
Regarded as an element of $Q_R$ the Casimir element has form
$C=2e+\sum_{i=1}^s\,y_iz_i+C'$ where $C'$ lies in $U(\g_R(0))$ and $y_1,\ldots, y_s\in\g(1)$; see \cite[p.~578]{Pr14}.

 Let $U_0(\g_R,e)$ denote the unital $R$-subalgebra of $U(\g_R,e)$ generated by
$\Theta(\g_R^e(0))$. 
It is immediate from Lemma~\ref{L:JEMS} and Theorem~\ref{T:part(i)} that $U_0(\g_R,e)\cong U(\g_R^e(0))$ as $R$-algebras. We denote by $U_0^+(\g_R,e)$ the augmentation ideal of $U_0(\g_R,e)$.
The preceding remarks together with Theorem~\ref{T:part(i)}(i) imply that  $C-2\Theta(e)$ lies in $U_0(\g_R,e)$ and has Kazhdan degree $\le 4$ (one should keep in mind here that $C-2\Theta(e)$ is fixed by the involution $\sigma\in\Aut U(\g_\C,e)$ defined in \cite[2.2]{Pr07a} and hence cannot have nonzero terms proportional to $\Theta(x_i)$ with  $x_i\in \g^e(1)$).
It follows that there exist $C_0\in U_0^+(\g_R,e)$ and $c\in R$ such that
\begin{eqnarray}
\label{e:Casimirequation}
C=2\Theta(e)+C_0+c\,1.
\end{eqnarray}

It is proved in \cite{Pr14} by direct computations that $U(\g_\C)$ contains a primitive ideal $\tilde{I}$ whose associated variety coincides with the Zariski closure of the adjoint orbit $\O_\C(I,J)$ and appears with multiplicity $1$ in the associated cycle of $\tilde{I}$. Furthermore, it is shown in {\it loc.\,cit.} that there is an irreducible highest weight module $L(\lambda)$ for $\g_\C$ such that
$\lambda(h_{\alpha_i})\in R$ for all $1\le i\le l$ and $\tilde{I}={\rm Ann}_{U(\g_\C)}\,L(\lambda)$. It is well-known that the Casimir element $C$ acts on $L(\lambda)$ 
as the scalar operator with eigenvalue $q=\langle \lambda,\lambda+2\rho\rangle$ where $\rho$ is the half-sum of all roots in $\Phi_+(\Pi)$ and $\langle\,\cdot\,,\cdot\,\rangle$ is 
the symmetric bilinear form on $\t_\C^*$ induced by the restriction of $\kappa$ to $\t_\C$; see \cite[Ch.~VIII, \textsection 6, No.~4]{B}, for example. Our choice of $\kappa$ in Subsection~\ref{ss:geR} implies that $q\in R$.
\begin{Proposition}\label{P:idealI2}
Suppose $e$ is listed in Table~1 and let $I_R$ be the two-sided ideal of $U(\g_R,e)$ generated by $\Theta(e)-\frac{q-c}{2}\,1$ and all commutators $[u,v]$ with $u,v\in U(\g_R,e)$. Then $I_R$ is a free $R$-module and $U(\g_R,e)=R\,1\oplus I_R$.	
\end{Proposition}
\begin{proof}
Combining \cite{Lo2, Pr10, gi} with Losev's multiplicity formula proved in \cite{Lo5}
we obtain that $$\tilde{I}\,=\,{\rm Ann}_{U(\g_\C)}\,\big(Q_\C\otimes_{U(\g_\C,e)}\,V\big)$$
for some one-dimensional $U(\g_\C, e)$-module $V$. Let $I_\C$ denote the annihilator of $V$ in $U(\g_\C,e)$. This is a two-sided ideal of codimension $1$ in $U(\g_\C,e)$ which contains $C-q\,1$ and all commutators
$[u,v]$ with $u,v\in U(\g_R,e)$. By Proposition~\ref{P:R-derived}(iii), $\g_R^e(0)$ is perfect. In conjunction with Lemma~\ref{L:JEMS} this yields $C_0\in I_\C$. By \eqref{e:Casimirequation} we have $\Theta(e) -\frac{q-c}{2} - C_0 \in I_\C$ and $I\subseteq I_\C\cap U(\g_R,e)$ follows.
 
We claim that $I_R=I_\C\cap U(\g_R,e)$ and $U(\g_R, e)\,=\,R\,1\oplus I_R$.
Due to Theorem~\ref{T:part(i)}(i), the claim will follow if we show that for every $x\in\g_R^e(n_k)$ there exists $c_x\in R$ such that
$\Theta(x)-c_x\,1\in I_R$. 
 We again argue by induction on $n_k$.
Since Proposition~\ref{P:R-derived}(iii) entails  
that $\g_R(i)=[\g_R(0),\g_R(i)]$ for $i\in\{0,1\}$  Lemma~\ref{L:JEMS} yields $\Theta(\g_R(0))\cup \Theta(\g_R(1))\subset I_R$. 
Therefore, our claim holds for $n_k\le 1$.

 Suppose that $n_k\ge 2$ and for every $y\in \g_R^e(n_i)$ with $n_i<n_k$ there exists $c_y\in R$ such that $\Theta(y)-c_y\,1\in I$. Let $x\in\g_R^e(n_k)$.
By Proposition~\ref{P:R-derived}(iii), 
$x=  \mu e+\sum_{i=1}^t\,[u_i,v_i]$ for some $\mu\in R$, $u_i\in\g_R^e(a_i)$, $v_i\in\g_R^e(b_i)$, where $a_i, b_i\in \Z_+$ and $a_i+b_i=n_k$. We again  apply the argument from the proof of Proposition~\ref{P:Theta-k}.
As our present setting is a bit different,
at each step of the reduction process we use $\Theta(y)-c_y\,1$ with $y\in [\g_R^e,\g_R^e](n_i)$ and $\Theta(e)-\frac{q-c}{2}\,1$ instead of 
$\Theta(y)$ and $\Theta(e)$. As in the proof of Proposition~\ref{P:idealI1} the iterations will stop at an element of $\F_0 U(\g_R,e) = \k$, implying that
$\Theta(x)-c_x\,1\in I_R$ for some $c_x\in R$.
This completes the proof.
 \end{proof}
\subsection{Final remarks and open problems} Theorem~\ref{main1}(3) is an immediate consequence of Theorem~\ref{T:part(i)}(iii) and Propositions~\ref{P:idealI1} and \ref{P:idealI2}. Theorem~\ref{main} follows from  Proposition~\ref{P:reductionrigid}, Proposition~\ref{P:reductionprop} and Theorem~\ref{main1}. 
\begin{Remark}
The argument used in the proof of Proposition~\ref{P:idealI1} shows that
if $e=e_{I,J}$ is rigid and not listed in Table~1 then  
$U_\chi(\g_\k,e_\k)$ affords a unique one-dimensional representation.
Hence the reduced enveloping algebra $U_{\chi}(\g_\k)$
has a unique module of dimension $p^{d(\chi)}$.
\end{Remark}
\begin{Remark}\label{rem}
If $e=e_{I,J}$ is rigid and special in the sense of Lusztig then $e$ is not listed in Table~1. For such an element one 
can determine the central character of the unique $U_\chi(\g_\k)$-module of dimension $p^{d(\chi)}$ by using the Arthur--Barbasch--Vogan recipe (with a subsequent reduction modulo $p$). Indeed, let $(e^\vee,h^\vee, f^\vee)$ be an $\sl_2$-triple in the Langlands dual Lie algebra $\g_\C^\vee$ such that $e^\vee\in\g_\C^\vee$ corresponds to $e$ under the Lusztig--Spaltenstein duality map and $h^\vee$ lies in the Weyl chamber of $\t^*_\mathbb{R}$ associated with $\Pi$. By \cite[Theorem~B]{Pr14}, the unique multiplicity-free primitive ideal of $U(\g_\C)$ whose associated variety coincides with the Zariski closure of $\O_\C(I,J)$ has form $\tilde{I}={\rm Ann}_{U(\g_\C)}\,L(\lambda)$ where $\lambda=\frac{1}{2}h^\vee-\rho$. It is well-known that $\lambda\in \t_\Z^*$.
By Losev's multiplicity formula, mentioned earlier, there exists a one-dimensional $U(\g_\C,e)$-module $V_0$ such that  $\tilde{I}={\rm Ann}_{U(\g_\C)}\big(Q_\C\otimes_{U(\g_\C,e)} V_0\big)$. 
Recall that the centre of $U(\g_\C,e)$ identifies canonically with $Z(U(\g_\C))$; see \cite[p.~524]{Pr07a}, for example. 
Let $I_\C$ be 
the unique two-sided ideal of codimension $1$ in $U(\g_\C,e)$.
Then
$V_0\,\cong\,U(\g_\C,e)/I_\C$ as $U(\g_\C,e)$-modules.
The proof of Proposition~\ref{P:idealI1} shows that the augmentation ideal $I_R$ of $U(\g_R,e)$ coincides with $I_\C\cap U(\g_R,e)$ and $I_\k=I_R\otimes_R \k$ is the only ideal of codimension $1$ in $U(\g_\k,e_\k)$. It follows that $I_\k$ contains the ideal $J_\chi$ of $Z_p(\g_\k,e_\k)$ implying that $$V_{0\,,\k}:=\,(U(\g_R,e)/I_R)\otimes_R\k\,\cong\,U(\g_\k,e_\k)/I_\k$$ is a $1$-dimensional $U_\chi(\g_\k,e_\k)$-module.
The above discussion shows that  $I_\C\cap Z(U(\g_\C))$ 
 coincides with the kernel of the central character corresponding to the $W(\Phi)$-orbit of $\frac{1}{2}h^\vee\in\t_\Z^*$. 
As a consequence, the Harish-Chandra centre $U(\g_\k)^{\Ad\,G_\k}\cong\, U(\g_\Z)^{\Ad\,G_\Z}\otimes_\Z \k$ acts on the unique small $U_\chi(\g_\k)$ module $V=\,Q_{\k,\,\chi}\otimes_{U_\chi(\g_\k,e_\k)}\,V_{0,\,\k} $ through the character associated with the $W(\Phi)$-orbit of the image of 
$\frac{1}{2}h^\vee$ in $\t^*_{\mathbb{F}_p}=\,\t^*_\Z\otimes_\Z\mathbb{F}_p$.
Since the centre of $U(\g_\k)$ is generated by $U(\g_\k)^{\Ad\,G_\k}$ and $Z_p(\g_\k)$ this determines
the action of the centre of $U(\g_\k)$ on $V$. We refer to \cite[9.6]{Ja98} and references therein for more detail on the structure of the centre of $U(\g_\k)$.
\end{Remark}
\begin{Remark}
	Combining Proposition~\ref{P:arprop} and Lemma~\ref{L:JEMS} with the proof of Theorem~\ref{T:part(i)} it is straightforward to see that for $p>2$ the first part of Theorem~\ref{main1} is true for all almost rigid nilpotent elements $e$ in the orthogonal and symplectic Lie algebras. In order to prove the second part of Theorem~\ref{main1} in this case it would be important to determine the central character of the unique one-dimensional $U(\g_\C,e)$-module stable under the action of the component group of $G_\C^e$. At the time of writing this seems to be an open problem for {\it non-special} almost rigid nilpotent elements in $\so_N(\C)$ and $\sp_{N}(\C)$.
\end{Remark}

\begin{Remark}
Suppose $\g=\g_\k$ and $\chi=\kappa(e_\k,-)$.	
Following \cite{PS99} we denote by $\z$ the coadjoint stabiliser $\g^\chi$ and write $\F(\g,\z)$ for the (restricted) coadjoint $\g$-module
${\rm Hom}_{U_0(\z)}(U_0(\g),\k)$. The comultiplication
of $U(\g)$ endows $\F(\g,\z)$ with a commutative associative $\k$-algebra structure and the Lie algebra $\g_\k$ acts on $\F(\g,\z)$ as derivations. Let $S_\chi(\g)$ denote the quotient of the symmetric algebra $S(\g)$ by its ideal generated by all $(x-\chi(x))^p$ with $x\in \g$. Clearly, $S_\chi(\g)$ is a local $\k$-algebra of dimension $p^{\dim \g}$ and $\g$ acts on $S_\chi(\g)$ as derivations. 
By \cite[Proposition~3.4 and Theorem~3.2]{PS99}, the algebra $S_\chi(\g)$ has a unique maximal $\g$-invariant ideal $\I$ and 
$S_\chi(\g)/\I\,\cong\,
\F(\g,\z)$ as $\k$-algebras and $\g$-modules.  
The Lie--Poisson structure of $S(\g)$ induces a natural Poisson structure on $S_\chi(\g)/\I$ and the Lie algebra $\g$ acts on $S_\chi(\g)/\I$ by Poisson derivations.	

Combining the above with Theorem~\ref{main} and \cite[\textsection 2]{PS99} we can prove that there is an associative $\k[t]$-algebra $\mathcal{A}$ which enjoys the following properties:
\begin{itemize}	
\item[(1)\,]	 $\mathcal{A}$ is a free $\k[t]$-module of rank $p^{2d(\chi)}$;

\smallskip
	
\item[(2)\,] $\mathcal{A}_c:=\mathcal{A}/(t-c)\mathcal{A}$ is isomorphic to
the matrix algebra ${\rm Mat}_{p^{d(\chi)}}(\k)$ for every $c\in \k^\times$;

\smallskip

\item[(3)\,] $\mathcal{A}_0:= \mathcal{A}/t\mathcal{A}$ is isomorphic $\F(\g,\z)$ as $\k$-algebras and $\g$-modules;

\smallskip

\item[(4)\,] $\g$ acts on each matrix $\k$-algebra $\mathcal{A}_c$ with $c\ne 0$ by inner derivations.

\smallskip

\item[(5)\,] $\g$ acts on $\mathcal{A}_0$ by Poisson derivations.

\end{itemize}
Our proof, outlined below,  makes use of some properties of the $\k[t]$-algebra $U_{\chi,\,t^2}(\g)$ introduced in \cite[\textsection 2]{PS99}.	
More precisely, it relies on 
the fact that the above-mentioned Poisson structure on $S_\chi(\g)$  
can be obtained by taking commutators in $U_{\chi,\,t^2}(\g)$, dividing the result by $t^2$ and picking the remainder in $U_{\chi,\,t^2}(\g)/tU_{\chi,\,t^2}(\g)\cong S_\chi(\g)$.

In order to obtain $\mathcal{A}$ one invokes the diagonalisable automorphism $\tau$ of the $\k[t,t^{-1}]$-algebra $U_{\chi,\,t^2}(\g)[t,t^{-1}]$ such that $\tau(x)=t^{i+2}x$ for all $x\in \g(i)\subset U_{\chi,\,t^2}(\g)$ (and all $i\in \Z$) and applies it to the annihilator of a small
$U_\chi(\g)$-module regarded as a $\k$-subspace of  $U_{\chi,\,t^2}(\g)$, call it $\mathcal{J}_\chi$. The discussion in \cite[2.1]{PS99} shows that this way one obtains a two-sided ideal
$\mathcal{J}_{\chi,\,t}$ of  $U_{\chi,\,t^2}(\g)[t,t^{-1}]$.
Using a $\k$-basis of $\mathcal{J}_\chi$ compatible with the (finite)  Kazhdan $\Z$-filtration of $U_\chi(\g)$
it is straightforward to see that the $\k[t]$-module 
$\mathcal{J}_{\chi,\,t}\cap U_{\chi,\,t^2}(\g)$ is a direct summand of $U_{\chi,\,t^2}(\g)$.
As a result, the $\k[t]$-algebra $$\mathcal{A} = U_{\chi,\,t^2}(\g)/\big(\mathcal{J}_{\chi,\,t}\cap U_{\chi,\,t^2}(\g)\big)$$ arises as a quantisation of the Poisson algebra
$\F(\g,\z)$. We stress that the construction of $\mathcal {A}$ depends on the choice of a small
$U_\chi(\g)$-module.

Let $M$ be the maximal ideal of
the coset of identity in $G/G^\chi$ and write $M^{(p)}$ for the ideal of the coordinate ring $\k[G/G^\chi]$ generated by all $\varphi^p$ with $\varphi\in M$. When $\g^\chi=\Lie(G^\chi)$ (which holds in our setting) we have that
$\F(\g,\z)\cong\k[G/G^\chi]/M^{(p)}$ as $\k$-algebras and $\g$-modules.
This indicates that small representations with $p$-character $\chi$ might play a role in the problem of quantising the function algebras on 
homogeneous spaces $G/G^\chi$ in good characteristic and (possibly) over $R$.
\end{Remark}


\begin{Remark}
It would be interesting to obtain more explicit realisations of small $U_\chi(\g_\k)$-modules. In principle, this can be done by following the procedure described in \cite[\textsection 3]{Pr07b}. 

We start with $\lambda\in \t_R^*$ and observe that the irreducible highest weight module $L_\C(\lambda)$ for the Lie algebra $\g_\C$ has a nice $R$-form, 
$L_R(\lambda)$, generated over $U(\g_R)$ by a highest weight vector $v_0$ of $L_\C(\lambda)$; see \cite[2.2]{Pr07b}. Let  
$L_p(\lambda):=L_R(\lambda)\otimes_R\k$. If $\lambda$ is not an integral dominant weight then the $\g_\k$-module $L_p(\lambda)$ is infinite-dimensional, but
this can be remedied by picking
a linear function $\chi\in\g_\k^*$ and letting
$$L^\chi_p(\lambda)\,:=\,L_p(\lambda)
\otimes_{Z_p(\g_\k)}\k_\chi.$$
 Let $\bar{v}_0$ denote the image of $v_0$ in $L_p(\lambda)$ and write $J_p(\lambda)$ for the annihilator of $\bar{v}_0$ in $Z_p(\g_\k)$. Let ${\rm V}_p(\lambda)\subset \g^*_\k$
denote the zero locus of the ideal $J_p(\lambda)$. By \cite[Lemma~3.1]{Pr07b}, the $U_\chi(\g_\k)$-module $L^\chi_p(\lambda)$ is nonzero for every $\chi\in {\rm V}_p(\lambda)$.

Suppose that the open orbit of the associated variety of $I(\lambda)={\rm Ann}_{U(\g_\C)}\, L_{\C}(\lambda)$ contains a nilpotent element $e_{I,J}$ and write $\O_\k(I,J)$ for the corresponding nilpotent orbit in $\g_\k$. Let
${\rm V}_p^{\circ}(\lambda)$ denote the set of all $\chi\in {\rm V}_p(\lambda)$ such that $\chi=\kappa(e,-)$
for some $e\in \O_\k(I,J)$. 
By \cite[Corollary~3.1]{Pr07b}, there exists a positive integer $D=D(\lambda)$ such that for every prime $p> D$ the set ${\rm V}^\circ_p(\lambda)$ is nonempty. Furthermore, 
if $\chi\in {\rm V}_p^\circ(\lambda)$ then $\dim L^\chi_p(\lambda)=kp^{d(\chi)}$ for some $k\le D$.


Let $\chi=\kappa(e,-)$ and suppose $p$ is a good prime for $G_\Z$. 
Let $\Lambda_{I,J}$ denote the set of all $\lambda\in \t_R^*$ such that the primitive ideal $I(\lambda)$ coincides with the annihilator of the irreducible $\g_\C$-module $Q\otimes_{U(\g_\C,\, e_{I,J})}\C_\eta$ where $\eta$ is a one-dimensional representation of $U(\g_\C, e_{I,J})$.
We conjecture that for a suitable choice of  $e\in \O_\k(I,J)$ any small $U_\chi(\g_\k)$-module 
appears as a composition factor of one of the modules 
$L_p^\chi(\lambda)$ with $\lambda\in \Lambda_{I,J}$ and $\chi\in {\rm V}_p^\circ(\lambda)$. If this conjecture were true one would be able to 
study $U_\chi(\g)$-modules under mild assumptions on $p$ by applying to small modules various translation functors.
\end{Remark}

\end{document}